\documentclass[12pt,a4paper,reqno,oneside]{amsart}
 
\usepackage[utf8]{inputenc}
\usepackage[T1]{fontenc}
\usepackage{lmodern}
\usepackage{amsmath}
\usepackage{amssymb}
\usepackage{amsthm}
\allowdisplaybreaks
\usepackage{geometry}
\usepackage{soul}
\usepackage{graphicx}
\usepackage{bbm}
\usepackage{float}
\usepackage{enumerate}
\usepackage{scrhack}
\usepackage{hyperref}
\usepackage{cleveref}

\usepackage{fancyhdr}

\newtheoremstyle{theorem}	
   {}       			
   {}      				
   {\itshape}  				
   {\parindent} 			
   {\bfseries} 				
   {}         				
   {.7em}      				
   {}          				

\newtheoremstyle{definition}
  {}       				
  {}      				
  {} 		 				
  {\parindent} 				
  {\bfseries} 				
  {}         				
  {.7em}      				
  {}          				
  
\newtheoremstyle{remark}
  {}       			
  {}      			
  {} 					
  {\parindent} 			
  {\itshape} 			
  {.}         			
  {.7em}      			
  {}          			

\theoremstyle{theorem}
\newtheorem{theorem}{Theorem}[section]
\newtheorem{corollary}[theorem]{Corollary}
\newtheorem{lemma}[theorem]{Lemma}

\theoremstyle{remark}
\newtheorem{remark}[theorem]{Remark}

\theoremstyle{definition}

\newtheorem{definition}[theorem]{Definition}

\newcommand{\Bb}{\mathbb{B}}

\newcommand{\Ebb}{\mathbb{E}}
\newcommand{\Fbb}{\mathbb{F}}

\newcommand{\Hbb}{\mathbb{H}}
\newcommand{\Nbb}{\mathbb{N}}
\newcommand{\Pbb}{\mathbb{P}}
\newcommand{\Qbb}{\mathbb{Q}}
\newcommand{\Rbb}{\mathbb{R}}

\newcommand{\Bcal}{\mathcal{B}}
\newcommand{\Ccal}{\mathcal{C}}
\newcommand{\Ecal}{\mathcal{E}}
\newcommand{\Fcal}{\mathcal{F}}

\newcommand{\Hcal}{\mathcal{H}}
\newcommand{\Kcal}{\mathcal{K}}
\newcommand{\Mcal}{\mathcal{M}}

\newcommand{\Pcal}{\mathcal{P}}
\newcommand{\Scal}{\mathcal{S}}

\newcommand{\Xcal}{\mathcal{X}}
\newcommand{\Ycal}{\mathcal{Y}}

\newcommand{\Kfrak}{\mathbf{K}}
\newcommand{\kfrak}{\mathfrak{K}}

\DeclareMathOperator*{\Var}{Var}

\newcommand{\Ep}[2]{\Ebb\negthinspace\left[\left\vert #1 \right\vert^{#2} \right]^{\frac{1}{#2}}\negthinspace}
\newcommand{\EpO}[2]{\Ebb\negthinspace\left[\left\vert #1 \right\vert^{#2} \right]\negthinspace}
\newcommand{\Eabs}[1]{\Ebb\negthinspace\left[\left\vert #1 \right\vert \right]\negthinspace}
\newcommand{\EW}[1]{\Ebb\negthinspace\left[ #1 \right]\negthinspace}

\newcommand{\BL}{\text{BL}}
\newcommand{\Lip}{\text{Lip}}

\begin{document}


\title[MKV equations on infinite-dimensional Hilbert spaces]{McKean-Vlasov equations on infinite-dimensional Hilbert spaces with irregular drift and additive fractional noise}
\author[M.Bauer]{Martin Bauer}
\address{M. Bauer: Department of Mathematics, LMU, Theresienstr. 39, D-80333 Munich, Germany.}
\email{bauer@math.lmu.de}
\author[T. Meyer-Brandis]{Thilo Meyer-Brandis}
\address{T. Meyer-Brandis: Department of Mathematics, LMU, Theresienstr. 39, D-80333 Munich, Germany.}
\email{meyerbra@math.lmu.de}
\date{\today}
\maketitle


\begin{center}
\parbox{13cm}{
\begin{footnotesize}
\textbf{\textsc{Abstract.}} This paper establishes results on the existence and uniqueness of solutions to McKean-Vlasov equations, also called mean-field stochastic differential equations, in an infinite-dimensional Hilbert space setting with irregular drift. Here, McKean-Vlasov equations with additive noise are considered where the driving noise is cylindrical (fractional) Brownian motion. The existence and uniqueness of weak solutions are established for drift coefficients that are merely measurable, bounded, and continuous in the law variable. In particular, the drift coefficient is allowed to be singular in the spatial variable. Further, 
we discuss existence of a pathwisely unique strong solution as well as Malliavin differentiability.\\[0.2cm]
\end{footnotesize}
}
\end{center}

\section{Introduction}\label{sec:introduction}

	Throughout the paper let $T>0$ be a finite time horizon and let $(\Omega, \Fcal, \Fbb, \Pbb)$ be a complete filtered probability space. McKean-Vlasov (for short MKV) equations, also called mean-field stochastic differential equations, are an extension of stochastic differential equations, where the coefficients in addition to time and space are depending on the law of the solution. More precisely, a finite-dimensional McKean-Vlasov equation is commonly defined as
	\begin{align}\label{eq:MFSDEgeneral}
		dX_t = b\left(t, X_t, \Pbb_{X_t} \right) dt + \sigma\left(t, X_t, \Pbb_{X_t} \right) dB_t,~t\in [0,T],~X_0 = x \in \Rbb^d,
	\end{align}
	where $b:[0,T] \times \Rbb^d \times \Pcal_1(\Rbb^d) \to \Rbb^d$ and $\sigma:[0,T] \times \Rbb^d \times \Pcal_1(\Rbb^d) \to \Rbb^{d \times n}$ are measurable functions, $\Pcal_1(\Rbb^d)$ is the set of probability measures over $\Rbb^d$ with finite first moment, $(\Pbb_{X_t})_{t\in[0,T]}$ denotes the law of $(X_t)_{t\in[0,T]}$ under the probability measure $\Pbb$, and $B= (B_t)_{t\in[0,t]}$ is $n$-dimensional Brownian motion. 
	
	The field of MKV equations is a research area that currently gains broad attention. Developing historically from the works of Vlasov \cite{Vlasov_VibrationalPropertiesofElectronGas}, Kac \cite{Kac_FoundationsOfKineticTheory}, and McKean \cite{McKean_AClassOfMarkovProcess} on the modeling of particle systems in mathematical physics, an increased interest in MKV equations emerged following the work of Lasry and Lions \cite{LasryLions_MeanFieldGames} who applied the mean-field approach to topics in Economics and Finance. Later Carmona and Delarue transfered this approach on mean-field games to a probabilistic environment, cf. the manuscript \cite{CarmonaDelarue_Book} and the cited sources therein.

	In this paper we extend the finite-dimensional MKV equation \eqref{eq:MFSDEgeneral} to infinite dimensions and further consider cylindrical fractional Brownian motion as additive driving noise, i.e. we look at MKV equations of the form
	\begin{align}\label{eq:mfsde}
		X_t = x + \int_0^t b(s,X_s,\Pbb_{X_s}) ds + \Bb_t, \quad t\in[0,T], \quad x \in \Hcal,
	\end{align}
	on a separable Hilbert space $\Hcal$. Here, $\Bb = ( \Bb_t )_{t\in[0,T]}$ is (weighted) cylindrical fractional Brownian motion defined as
	\begin{align*}
		\Bb_t = \sum_{k\geq 1} \lambda_k B_t^{H_k} e_k, \quad t\in [0,T],
	\end{align*}
	where $\lambda = \lbrace \lambda_k \rbrace_{k \geq 1} \in \ell^1$, $\lbrace e_k \rbrace_{k\geq 1}$ is an orthonormal basis of $\Hcal$, and $\lbrace B^{H_k} \rbrace_{k\geq 1}$ a sequence of independent one-dimensional fractional Brownian motions with Hurst parameters $\Hbb := \lbrace H_k \rbrace_{k\geq 1} \subset (0,1)$. Note that Hurst parameters in the entire range $(0,1)$ are admitted, and we introduce the following partition: $I_- := \lbrace k: H_k \in (0,1/2) \rbrace$, $I_0 := \lbrace k: H_k = 1/2 \rbrace$, and $I_+ := \lbrace k: H_k \in (1/2,1) \rbrace$. The main objective of this paper is to study existence and uniqueness of a solution to the infinite-dimensional MKV equation \eqref{eq:mfsde} for irregular drift coefficients $b$.
	
	 In the literature existence and uniqueness of solutions of the finite-dimensional MKV equation \eqref{eq:MFSDEgeneral} is examined in several papers with respect to various assumptions on the coefficients $b$ and $\sigma$, c.f. \cite{Bauer_RegularityOfMFSDE}, \cite{Bauer_MultiDim}, \cite{Bauer_StrongSolutionsOfMFSDEs}, \cite{BuckdahnDjehicheLiPeng_MFBSDELimitApproach}, \cite{BuckdahnLiPeng_MFBSDEandRelatedPDEs}, \cite{BuckdahnLiPengRainer_MFSDEandAssociatedPDE}, \cite{Chiang_MKVWithDiscontinuousCoefficients}, \cite{de2015strong}, \cite{JourdainMeleardWojbor_NonlinearSDEs}, \cite{LiMin_WeakSolutions}, \cite{mahmudov2006mckean}, and \cite{MishuraVeretennikov_SolutionsOfMKV}. In particular, in \cite{LiMin_WeakSolutions} Li and Min show the existence of a weak solution of a path dependent finite-dimensional MKV equation by the means of Girsanov's theorem and Schauder's fixed point theorem, where they assume that $b$ is merely measurable and bounded as well as continuous in the law variable. Further, uniqueness in law is proven under the additional assumption that $b$ admits a modulus of continuity. Mishura and Veretennikov show in \cite{MishuraVeretennikov_SolutionsOfMKV} inter alia the existence of a pathwise unique strong solution to a finite-dimensional MKV equation \eqref{eq:MFSDEgeneral}, where they assume the drift coefficient $b$ to be merely measurable, of at most linear growth, and continuous in the law variable in the topology of weak convergence. For their proof they use an approximational approach based on techniques applied by Krylov in the theory of stochastic differential equations, cf. \cite{krylov1969ito}. In \cite{Bauer_MultiDim}, we consider MKV equation \eqref{eq:MFSDEgeneral} with additive noise, i.e. $\sigma \equiv 1$, and singular drift coefficients $b$. More precisely, for $b$ being bounded and continuous in the law variable with respect to the Kantorovich-Rubinstein metric, it is shown that there exists a Malliavin differentiable strong solution of MKV equation \eqref{eq:MFSDEgeneral}.  For one-dimensional solutions of \eqref{eq:MFSDEgeneral} we even allow for certain linear growth behavior of the drift in \cite{Bauer_StrongSolutionsOfMFSDEs}.

	Using similar approaches as in \cite{Bauer_MultiDim} and \cite{Bauer_StrongSolutionsOfMFSDEs}, in this paper existence of a weak solution to the infinite-dimensional MKV equation \eqref{eq:mfsde} is established under the assumption that the drift coefficient $b$ is in the space $L^\infty(\Hcal)$, i.e. there exists a sequence $C \in \ell^1$ such that $\Vert b_k \Vert_\infty \leq C_k$ for every $b_k := \langle b, e_k \rangle_\Hcal$, $k\geq 1$, and for $k\in I_+$ the projection of the drift $b_k$ is Hölder continuous, i.e.
	\begin{align*}
		\vert b_k (t,x,\mu) - b_k (s,y,\nu) \vert \leq C_k \left( \vert t-s \vert^{\gamma_k} + \Vert x-y \Vert_\Hcal^{\alpha_k} + \Kcal(\mu, \nu)^{\beta_k} \right),
	\end{align*}
	for suitable constants $C_k, \gamma_k, \alpha_k, \beta_k >0$, and $\Kcal$ denotes the Kantorovich-Rubinstein metric, cf. \eqref{eq:KantorovichMetric}. For $k\in I_-\cup I_0$ it is assumed that the projection $b_k$ is merely continuous with respect to the law variable. More precisely, in order to show existence of a weak solution we first apply Girsanov's theorem to show the existence of a weak solution to the stochastic differential equation, for short SDE,
	\begin{align*}
		dX_t^\mu = b\left(t, X_t^\mu, \mu_t \right)dt + d\Bb_t, ~t\in[0,T], ~X_0 = x \in \Hcal,
	\end{align*}
	where $\mu \in \Ccal([0,T];\Pcal_1(\Hcal))$ is an arbitrary measure process continuous with respect to time. Afterwards Schauder's fixed point theorem \cite{Schauder} is applied to the function 
	\begin{align*}
		\varphi(\mu) = \Pbb_{X_t^\mu}
	\end{align*}
	to show the existence of a fixed point and in particular, to conclude existence of a weak solution to MKV equation \eqref{eq:mfsde}. \par 
	Assuming additionally that the drift coefficient $b$ is Lipschitz continuous in the law variable, it is shown that the solution of the infinite-dimensional MKV equation \eqref{eq:mfsde} is unique in law. In order to show uniqueness in law, we apply similar to \cite{Bauer_MultiDim} and \cite{Bauer_StrongSolutionsOfMFSDEs} Girsanov's theorem and a Grönwall type argument. \par
	Existence of a strong solution to MKV equation \eqref{eq:mfsde} is then a consequence of results on ordinary SDEs. Indeed, we can associate the following SDE to MKV equation \eqref{eq:mfsde}:
	\begin{align}\label{eq:SDEaux}
		dY_t = b^{\Pbb_X}\left(t, Y_t \right) dt + d\Bb_t,~t\in [0,T],~Y_0 = x \in \Hcal,
	\end{align}
	where $b^{\Pbb_X}\left(t, y \right) := b\left(t, y, \Pbb_{X_t} \right)$ and $X$ is a weak solution of MKV equation \eqref{eq:mfsde}. In order to show that \eqref{eq:mfsde} has a strong solution, it suffices to show that there exists a weak solution that is measurable with respect to the filtration generated by the driving noise $\Bb$. Since $X$ is as a weak solution to MKV equation \eqref{eq:mfsde} also a weak solution of SDE \eqref{eq:SDEaux}, it is sufficient to show that every weak solution $Y$ of SDE \eqref{eq:SDEaux} is a strong solution. Furthermore, if MKV equation \eqref{eq:mfsde} has a weakly unique solution, the associated SDE \eqref{eq:SDEaux} is uniquely determined and consequently, pathwise uniqueness of the solution $Y$ of SDE \eqref{eq:SDEaux} implies pathwise uniqueness of the solution $X$ of MKV equation \eqref{eq:mfsde}. Thus, applying existence results on SDEs as for example stated in \cite{banos2019infiniteSDE}, \cite{MeyerBrandisProske_ConstructionOfStrongSolutionsOfSDEs}, \cite{nualart2002regularization}, and \cite{veretennikov1981strong}, yields existence of a (pathwisely unique) strong solution of MKV equation \eqref{eq:mfsde}. Analogously, Malliavin differentiability of the solution to MKV equation \eqref{eq:mfsde} is deduced from results on SDEs, cf. \cite{Bauer_MultiDim} and \cite{Bauer_StrongSolutionsOfMFSDEs}.

	The paper is structured as follows. In \Cref{sec:framework} we give a brief introduction to measure spaces, fractional calculus, and fractional Brownian motion. After introducing the driving noise $\Bb$ and a version of Girsanov's theorem, we present in \Cref{sec:solution} the main results of this paper on existence and uniqueness of a weak solution to the infinite-dimensional MKV equation \eqref{eq:mfsde}. Concluding, existence of a unique strong solution to MKV equation \eqref{eq:mfsde} and Malliavin differentiability are discussed in \Cref{sec:strongSolution}.

\bigskip

\textbf{\underline{Notation:}}
	Subsequently, we give some of the most frequently used notations. Throughout the paper, let $\Hcal$ be a separable Hilbert space with scalar product $\langle \cdot, \cdot \rangle_\Hcal$ and orthonormal basis $\lbrace e_k \rbrace_{k\geq 1} \subset \Hcal$. Denote by $\Vert \cdot \Vert_{\Hcal}$ the induced norm on $\Hcal$ defined by $\Vert x \Vert_{\Hcal} := \langle x,x \rangle_\Hcal^{\frac{1}{2}}$, $x \in \Hcal$. For every $x \in \Hcal$ and $k\geq 1$ we denote by $x^{(k)} := \langle x, e_k \rangle_\Hcal$ the projection onto the subspace spanned by $e_k$. We denote by $b_k: [0,T] \times \Hcal \times \Pcal_1(\Hcal) \to \Rbb$, the projection of $b$ onto the subspace spanned by $e_k$, $k \geq 1$. Furthermore, we assume for technical reasons that without loss of generality $T\geq 1$.

 Let $(\mathcal{X},\Vert \cdot \Vert_{\mathcal{X}}), (\mathcal{Y},\Vert \cdot \Vert_{\mathcal{Y}})$ be two normed spaces.
\begin{itemize}
\item $L^p(\Xcal; \Ycal)$ denotes the space of functions $f: \Xcal \to \Ycal$ with existing $p$-th moment, i.e.
\begin{align*}
	\int_{\Xcal} \Vert f(x) \Vert_{\Ycal}^p dx < \infty.
\end{align*}
	If $\Xcal = [a,b]$ is an interval on the real line and $\Ycal = \Rbb$, we write $L^p[a,b]$.
\item $\Ccal^\kappa([0,T]; \Xcal)$, $\kappa >0$, is defined as the space of $\kappa$-Hölder continuous functions $f: [0,T] \to \Xcal$, i.e. for all $t,s \in [0,T]$
\begin{align*}
	\Vert f(t) - f(s) \Vert_\Xcal \leq \vert t-s \vert^\kappa.
\end{align*}
\item We denote by $\Lip_C(\Xcal; \Ycal)$, $C>0$ the space of $C$-Lipschitz continuous functions $f: \Xcal \to \Ycal$, i.e. for all $x_1,x_2 \in \mathcal{X}$
\begin{align*}
	\Vert f(x_1) - f(x_2) \Vert_\Ycal \leq C \Vert x_1,x_2 \Vert_\Xcal.
\end{align*}
\item For a function $f:\Xcal \to \Ycal$ define $\Vert f \Vert_{\Lip} := \inf \lbrace C>0 : f \in \Lip_C(\mathcal{X}; \mathcal{Y})\rbrace$ and $\Vert f \Vert_{\infty} :=  \sup_{x\in \Xcal} \Vert f(x) \Vert_{\Ycal}$. We define the bounded Lipschitz norm of $f$ as $\Vert f \Vert_{\BL} := \Vert f \Vert_{\infty} + \Vert f \Vert_{\Lip}$. We say $f \in \BL(\Xcal;\Ycal)$, if $\Vert f \Vert_\BL \leq 1$.
\item The Beta function $\beta$ is defined by
\begin{align*}
	\beta(x,y) = \int_0^1 t^{x-1} (1-t)^{y-1} dt.
\end{align*}
\item The Gamma function $\Gamma$ is defined by
\begin{align*}
	\Gamma(x) = \int_0^\infty t^{x-1} e^{-t} dt.
\end{align*}
\item We write $E_1(\theta) \lesssim E_2(\theta)$ for two mathematical expressions $E_1(\theta),E_2(\theta)$ depending on some parameter $\theta$, if there exists a constant $C>0$ not depending on $\theta$ such that $E_1(\theta) \leq C E_2(\theta)$.
\item Let $C= \lbrace C_k \rbrace_{k\geq 1}$ and $D= \lbrace D_k \rbrace_{k\geq 1}$ be two sequences. Then, we denote $\frac{C}{D} := \lbrace \frac{C_k}{D_k} \rbrace_{k\geq 1}$.
\end{itemize}

\section{Framework}\label{sec:framework}

\subsection{Measure Spaces}\label{sec:measureSpaces}

	For a general introduction to (probability) measures on metric spaces we refer the reader e.g. to \cite{araujo1980central}. Let $(\Scal,d)$ be a complete separable metric space, in particular, $(\Scal,d)$ is a Radon space. We define the space $\Mcal(\Scal)$ as the space of finite signed Radon measures on $(\Scal, \Bcal(\Scal))$, where $\Bcal(\Scal)$ is the Borel-$\sigma$-algebra on $\Scal$. Moreover, let
	\begin{align*}
		\Mcal_p(\Scal) := \left\lbrace \mu \in \Mcal(\Scal) : \int_{\Scal} d(x,x_0)^p \vert \mu \vert(dx) < \infty \text{ for some } x_0 \in \Scal \right\rbrace,
	\end{align*}
	be the set of finite signed Radon measures over $(\Scal, \Bcal(\Scal))$ with finite $p$-th moment. $\Mcal_1(\Scal)$ equipped with the \emph{Kantorovich norm} $\Vert \cdot \Vert_{\Kcal}$, also called \emph{dual bounded Lipschitz norm}, defined by
	\begin{align*}
		\Vert \mu \Vert_{\Kcal} := \sup \left\lbrace \int_{\Scal} f(x) \mu(dx) :  \Vert f \Vert_{\BL} \leq 1 \right\rbrace, \quad \mu \in \Mcal_1(\Scal),
	\end{align*}
	defines a separable Banach space. Analogously, define the according Kantorovich-Rubinstein metric $\Kcal$ by 
	\begin{align}\label{eq:KantorovichMetric}
		\Kcal(\mu,\nu) := \Vert \mu - \nu \Vert_{\Kcal}, \quad \mu,\nu \in \Mcal_1(\Scal).
	\end{align}
	Let $\Pcal_p(\Scal) \subset \Mcal_p(\Scal)$ be the set of probability measures over $(\Scal, \Bcal(\Scal))$ such that the $p$-th moment exists, i.e.
	\begin{align*}
		\Pcal_p(\Scal) := \left\lbrace \mu \in \Mcal_p(\Scal) : \mu(\Scal) = 1 \text{ and } \mu(A) \geq 0 \text{ for all } A \in \Bcal(\Scal) \right\rbrace.
	\end{align*}
	Lastly, define the set of continuous functions $\Ccal([0,T];\Mcal_1(\Scal))$ from the time interval $[0,T]$ to the space $\Mcal_1(\Scal)$ and equip it with the norm $\Vert \mu \Vert_{\Kcal^*} := \sup_{t\in [0,T]} \Vert \mu_t \Vert_{\Kcal}$, $\mu \in \Ccal([0,T];\Mcal_1(\Scal))$. It can be shown that $(\Ccal([0,T];\Mcal_1(\Scal)),\Vert \cdot \Vert_{\Kcal^*})$ is a linear separable Banach space.

\subsection{Fractional Calculus}\label{sec:fractionalCalculus}

	We give some basic definitions and properties on fractional calculus. For a general theory on this subject we refer the reader to \cite{oldham1974fractional}.

Let $f \in L^p [a,b]$ for some real numbers $a < b$, where $p\geq 1$, and let $\alpha>0$. The \emph{left--sided Riemann--Liouville fractional integral} is defined for almost all $x\in [a,b]$ by
\begin{align*}
	I_{a^+}^\alpha f(x) = \frac{1}{\Gamma (\alpha)} \int_a^x (x-y)^{\alpha-1}f(y)dy.
\end{align*}
Moreover, we denote by $I_{a^+}^{\alpha} (L^p[a,b])$ the image of $L^p[a,b]$ by the operator $I_{a^+}^\alpha$.

For $g \in I_{a^+}^{\alpha} (L^p[a,b])$ and $0<\alpha<1$, the \emph{left--sided Riemann--Liouville fractional derivative} is defined by
\begin{align}\label{eq:leftSidedFractionalDerivative}
	D_{a^+}^{\alpha} g(x)= \frac{1}{\Gamma (1-\alpha)} \frac{\partial}{\partial x} \int_a^x \frac{g(y)}{(x-y)^{\alpha}}dy.
\end{align}
The left--sided derivative of $g$ defined in \eqref{eq:leftSidedFractionalDerivative} can further be written as
\begin{align*}
	D_{a^+}^{\alpha} g(x)= \frac{1}{\Gamma (1-\alpha)} \left(\frac{g(x)}{(x-a)^\alpha}+\alpha\int_a^x \frac{g(x)-g(y)}{(x-y)^{\alpha+1}}dy\right).
\end{align*}

Similar to the fundamental theorem of calculus the following formulas hold
\begin{align*}
	I_{a^+}^\alpha (D_{a^+}^{\alpha} f) = f
\end{align*}
for all $f\in I_{a^+}^{\alpha} (L^p[a,b])$ and
\begin{align*}
	D_{a^+}^{\alpha}(I_{a^+}^\alpha  f) = f
\end{align*}
for all $f\in L^p[a,b]$.

\subsection{Fractional Brownian motion}

	In this section we recall the definition of a fractional Brownian motion and how it can be constructed from a standard Brownian motion using fractional calculus. For a more detailed introduction to this subject we refer the reader to \cite{biagini2008stochastic} and \cite[Chapter 5]{Nualart_MalliavinCalculus}
	
\begin{definition}
	We say $B^H = \left( B_t^H \right)_{t \in [0,T]}$ is a one-dimensional fractional Brownian motion with Hurst parameter $H \in (0,1)$, if it is a continuous and centered Gaussian process with covariance function
	\begin{align*}
		R_H(t,s) := \EW{B_t^H B_s^H} = \frac{1}{2}\left(t^{2H} + s^{2H} - |t-s|^{2H} \right).
	\end{align*}
\end{definition}

	It is well-known that $B^H$ has stationary increments and $(H-\varepsilon)$--Hölder continuous trajectories for all $\varepsilon>0$. Furthermore, $B^H$ is not a semimartingale and its increments are not independent for all $H\in (0,1)$ but $H= \frac{1}{2}$. For $H= \frac{1}{2}$ the process $B^H$ is a standard Brownian motion. \par
In the following we divide fractional Brownian motions into three classes by their Hurst parameters. The first class, $H \in (0,\frac{1}{2})$, is referred to as the \emph{singular case}, the second class, $H\in (\frac{1}{2}, 1)$, is referred to as the \emph{regular case}, and the third class, $H = \frac{1}{2}$, is the class of Brownian motions. Subsequently, we define for each class the kernels $K_H$ as well as the related operators $\Kfrak_H$ and $\Kfrak_H^{-1}$ which allow us to construct a fractional Brownian motion with Hurst parameter $H \in (0,1)$ from a standard Brownian motion. For more details see \cite{decreusefond1999stochastic} and \cite{nualart2002regularization}. Let $W = (W_t)_{t\in[0,T]}$ be a standard Brownian motion on the complete filtered probability space $(\Omega, \Fcal, \Fbb, \Pbb)$.

\vspace{0.5cm}

\emph{Singular Case:} Let $H \in (0, \frac{1}{2})$ and define the kernel
\begin{align}\label{eq:kernelSingular}
	K_H(t,s)= b_H \left[\left( \frac{t}{s}\right)^{H- \frac{1}{2}} (t-s)^{H- \frac{1}{2}} + \left( \frac{1}{2}-H\right) s^{\frac{1}{2}-H} \int_s^t u^{H-\frac{3}{2}} (u-s)^{H-\frac{1}{2}} du\right],
\end{align}
where $b_H = \sqrt{\frac{2H}{(1-2H) \beta(1-2H , H + \frac{1}{2})}}$. Then
\begin{align*}
	B_t^H := \int_0^t K_H(t,s) dW_s
\end{align*}
is a fractional Brownian motion with Hurst parameter $H$. Furthermore, the kernel $K_H$ yields an operator $\Kfrak_H: L^2[0,T] \to I_{0+}^{H+\frac{1}{2}}(L^2[0,T])$ defined by
\begin{align*}
	(\Kfrak_H f)(s) = \int_0^t K_H(t,s) f(s) ds = I_{0+}^{2H} s^{\frac{1}{2}-H} I_{0+}^{\frac{1}{2}-H} s^{H-\frac{1}{2}} f,
\end{align*}
where $f \in L^2[0,T]$. Finally, the inverse operator $\Kfrak^{-1}_H$ of $\Kfrak_H$ is defined by
\begin{align}\label{eq:isometrySingularInverse}
	\Kfrak_H^{-1} f = s^{\frac{1}{2} - H} D_{0+}^{\frac{1}{2}-H} s^{H-\frac{1}{2}} D_{0+}^{2H} f,
\end{align}
where $f \in I_{0+}^{H+\frac{1}{2}}(L^2[0,T])$. If $f$ is absolutely continuous, we can write 
\begin{align*}
	\Kfrak_H^{-1} f = s^{H- \frac{1}{2}} I_{0+}^{\frac{1}{2}-H} s^{\frac{1}{2}-H} f'.
\end{align*}

\vspace{0.5cm}

\emph{Regular Case:} Let $H \in (\frac{1}{2},1)$ and define the kernel
\begin{align}\label{eq:kernelRegular}
	K_H(t,s)= c_H s^{\frac{1}{2}-H} \int_s^t u^{H-\frac{1}{2}} ( u-s )^{H-\frac{3}{2}} du,
\end{align}
where $c_H = \sqrt{\frac{H(2H-1)}{\beta(2-2H , H - \frac{1}{2})}}$. Then
\begin{align*}
	B_t^H := \int_0^t K_H(t,s) dW_s
\end{align*}
is a fractional Brownian motion with Hurst parameter $H$. Furthermore, the kernel $K_H$ yields an operator $\Kfrak_H: L^2[0,T] \to I_{0+}^{H+\frac{1}{2}}(L^2[0,T])$ defined by
\begin{align*}
	(\Kfrak_H f)(s) = \int_0^t K_H(t,s) f(s) ds = I_{0+}^{1} s^{H - \frac{1}{2}} I_{0+}^{H - \frac{1}{2}} s^{\frac{1}{2} - H} f,
\end{align*}
where $f \in L^2[0,T]$. Finally, the inverse operator $\Kfrak^{-1}_H$ of $\Kfrak_H$ is defined by
\begin{align}\label{eq:isometryRegularInverse}
	\Kfrak_H^{-1} f = s^{H - \frac{1}{2}} D_{0+}^{H - \frac{1}{2}} s^{\frac{1}{2} - H} f',
\end{align}
where $f \in I_{0+}^{H+\frac{1}{2}}(L^2[0,T])$.

\vspace{0.5cm}

\emph{Brownian case:} Let $H = \frac{1}{2}$. Obviously, in the case $H=\frac{1}{2}$ the kernel is given by $K_H(t,s) \equiv 1$. Thus the operator $\Kfrak_H$ is defined as
\begin{align*}
	(\Kfrak_H f)(s) = \int_0^t K_H(t,s)f(s) ds = I_{0+}^{1} f,
\end{align*}
where $f \in L^2[0,T]$, and thus its inverse operator $\Kfrak^{-1}_H$ is given by
\begin{align}\label{eq:isometryBrownianInverse}
	\Kfrak_H^{-1} f = f',
\end{align}
where $f \in I_{0+}^{1}(L^2[0,T])$.

\vspace{0.5cm}

\begin{remark}
	Consider a sequence $\Hbb= \lbrace H_k \rbrace_{k\geq 1}$  of Hurst parameters. For the Hilbert space $\Hcal$ with basis $\lbrace e_k \rbrace_{k\geq 1}$ and $f \in L^2([0,T]; \Hcal)$, we define the operator $\Kfrak_\Hbb:  L^2([0,T]; \Hcal) \to I_{0+}^{\Hbb+1/2}(L^2([0,T]; \Hcal))$ componentwise by
\begin{align*}
	(\Kfrak_\Hbb f)(s) := \sum_{k\geq 1} (\Kfrak_{H_k} f_k )(s) e_k,
\end{align*}
where $f_k (s) := \langle f(s), e_k \rangle$, $k\geq 1$. Here, we say $f \in I_{0+}^{\Hbb+1/2}(L^2([0,T]; \Hcal))$, if for every $k\geq 1$ the projection $f_k$ is in $I_{0+}^{H_k+1/2}(L^2[0,T])$. Similarly we define the inverse $\Kfrak_\Hbb^{-1}$ of $\Kfrak_\Hbb$ by
\begin{align*}
	\Kfrak_\Hbb^{-1} f := \sum_{k\geq 1} \Kfrak_{H_k}^{-1} f_k e_k,
\end{align*}
where $f \in  I_{0+}^{\Hbb+1/2}(L^2([0,T]; \Hcal))$.
\end{remark}

\subsection{The weighted cylindrical fractional Brownian motion $\Bb$}\label{sec:wcfBm}
	Let us now define the driving noise $\Bb$ and afterwards derive a version of Girsanov's theorem for cylindrical fractional Brownian motion. Let $\lbrace W^{(k)} \rbrace_{k \geq 1}$ be a sequence of independent Brownian motions defined on the probability space $(\Omega, \Fcal, \Pbb)$. Similar to \cite{banos2019infiniteSDE} we define the \emph{cylindrical Brownian motion} $W:= \left( W_t \right)_{t\in[0,T]}$ taking values in $\Hcal$ by
\begin{align*}
	W_t := \sum_{k \geq 1} W_t^{(k)} e_k, \quad t \in [0,T].
\end{align*}
The natural filtration of $W$ augmented by the $\Pbb$-null sets is denoted by $\Fbb^W:=(\Fcal_t^W)_{t\in [0,T]}$.  Moreover, we consider a sequence of Hurst parameters $\Hbb= \lbrace H_k \rbrace_{k\geq 1}$ and the associated partition $\lbrace I_-, I_0, I_+ \rbrace$ of $\Nbb$ defined by
\begin{enumerate}[(i)]
	\item $k \in I_-: ~H_k \in \left( 0,\frac{1}{2} \right)$,
	\item $k \in I_0: ~H_k = \frac{1}{2}$,
	\item $k \in I_+: ~H_k \in \left(\frac{1}{2}, 1 \right)$.
\end{enumerate}
For $\lbrace H_k \rbrace_{k\geq 1}$ we construct the sequence of fractional Brownian motions $\lbrace B^{H_k} \rbrace_{k \geq 1}$ associated to $\lbrace W^{(k)} \rbrace_{k \geq 1}$ by
\begin{align*}
	B_t^{H_k} := \int_0^t K_{H_k}(t,s) dW_s^{(k)}, \quad t\in[0,T], \quad k\geq 1,
\end{align*}
where the kernel $K_{H_k}(\cdot,\cdot)$ is defined in \eqref{eq:kernelSingular} and \eqref{eq:kernelRegular}, respectively. Note that by construction the fractional Brownian motions $\lbrace B^{H_k}\rbrace_{k \geq 1}$ are independent. We then define the \emph{cylindrical fractional Brownian motion} $B^\Hbb$ with associated sequence of Hurst parameters $\Hbb = \lbrace H_k \rbrace_{k\geq 1}$ by
\begin{align*}
	B_t^\Hbb := \sum_{k\geq 1} B_t^{H_k} e_k, \quad t \in [0,T].
\end{align*}
Observe that the natural filtration of $B^\Hbb$ augmented by the $\Pbb$-null sets and $\Fbb^W$ coincide. Furthermore, for a given sequence $\lambda := \lbrace \lambda_k \rbrace_{k\geq 1} \in \ell^1$ such that $\sum_{k\in I_-} \frac{\lambda_k}{\sqrt{H_k}} < \infty$, we define the self-adjoint operator $Q: \Hcal \to \Hcal$ by
\begin{align*}
	Qx = \sum_{k\geq 1} \lambda_k^2 x^{(k)} e_k,
\end{align*}
and thereby construct the \emph{weighted cylindrical fractional Brownian motion} $\Bb$ by
\begin{align}\label{eq:qCylindricalFBm}
	\Bb_t := \sqrt{Q} B_t^\Hbb =  \sum_{k\geq 1} \lambda_k B_t^{H_k} e_k, \quad t\in[0,T].
\end{align}
Due to the following lemma, the process $\Bb$ is continuous in time and is in $L^2(\Omega; \Hcal)$.

\begin{lemma}\label{lem:continuity&L2}
	The weighted cylindrical fractional Brownian motion $\Bb$ defined in \eqref{eq:qCylindricalFBm} has almost surely continuous sample paths on $[0,T]$ and
	\begin{align*}
		\sup_{t \in [0,T]}\Ebb \left[ \left\Vert \Bb_t \right\Vert_{\Hcal}^2 \right] < \infty.
	\end{align*}
\end{lemma}
\begin{proof}
	Note first that for every $k \in I_-$ and time points $s,t\in[0,T]$, the fractional Brownian motion $B^{H_k}$ fulfills
	\begin{align*}
		\Ep{\left\vert B_t^{H_k} \right\vert - \left\vert B_s^{H_k} \right\vert}{2} \leq \Ep{B_t^{H_k} - B_s^{H_k}}{2} = \vert t-s \vert^{H_k}.
	\end{align*}		
	Hence due to \cite[Theorem 1]{borovkov2017bounds} the expected maximum of $\vert B^{H_k} \vert$ is bounded by 
	\begin{align*}
		\EW{\sup_{t\in[0,T]} \left\vert B_t^{H_k} \right\vert} = T^{H_k} \EW{\sup_{t\in[0,1]} \left\vert B_t^{H_k} \right\vert} \lesssim \frac{T^{H_k}}{\sqrt{H_k}}.
	\end{align*}
	In the case of a standard Brownian motion, i.e. $H=\frac{1}{2}$, the exact value of the expected maxima is known and is equal to $\sqrt{\frac{2 T}{\pi}}$. Using Sudakov-Fernique's inequality (see \cite[Theorem 1]{vitale2000some}) we thus get for $k\in I_0 \cup I_+$ the of $H_k$ independent upper bound
	\begin{align*}
		\EW{\sup_{t\in[0,T]} \left\vert B_t^{H_k} \right\vert} \leq T^{H_k - \frac{1}{2}} \EW{\sup_{t\in[0,T]} \left\vert W^{(k)}_t \right\vert} = T^{H_k} \sqrt{\frac{2}{\pi}} \leq T \sqrt{\frac{2}{\pi}}.
	\end{align*}
	Let us now consider the weighted cylindrical fractional Brownian motion $\Bb$ defined in \eqref{eq:qCylindricalFBm}. Using the previous bounds we have that
	\begin{align*}
		\EW{\sup_{t\in[0,T]} \Vert \Bb_t \Vert_\Hcal} &= \EW{\sup_{t\in[0,T]} \left\Vert \sum_{k\geq 1} \lambda_k B_t^{H_k} e_k \right\Vert_\Hcal} \leq \sum_{k\geq 1} \lambda_k \EW{\sup_{t\in[0,T]} \left\vert B_t^{H_k} \right\vert} \\
		&\lesssim \sum_{k\in I_-} \frac{\lambda_k T^{H_k}}{\sqrt{H_k}} + \sum_{k\in I_0 \cup I_+} \lambda_k \lesssim \sum_{k\in I_-} \frac{\lambda_k}{\sqrt{H_k}} + \Vert \lambda \Vert_{\ell^1} < \infty.
	\end{align*}
	Consequently, the stochastic process $\Bb$ is almost surely finite and the sequence of projections $\lbrace \sum_{k=1}^n \langle \Bb, e_k \rangle_\Hcal e_k \rbrace_{n\geq 1}$ is a Cauchy sequence in $L^1(\Omega; \Ccal([0,T];\Hcal))$ converging almost surely to the process $\Bb$. Thus, $t \mapsto \Bb_t$ is continuous on $[0,T]$. Furthermore, using Parseval's identity we get
	\begin{align*}
		\EW{\left\Vert \Bb_t \right\Vert_\Hcal^2} &= \EW{\left\Vert \sum_{k\geq 1} \lambda_k B_t^{H_k} e_k \right\Vert_\Hcal^2} = \sum_{k\geq 1} \lambda_k^2 \EW{ \left\vert B_t^{H_k} \right\vert^2} = \sum_{k\geq 1} \lambda_k^2 t^{2H_k} \leq  \Vert \lambda \Vert_{\ell^2}^2 T^2 < \infty.
	\end{align*}
\end{proof}

\subsection{Girsanov's theorem for cylindrical fractional Brownian motions}

Due to \cite[Theorem 2.2 and Remark 2.3]{banos2019infiniteSDE} we get the following version of Girsanov's theorem for cylindrical fractional Brownian motions.

\begin{theorem}[Girsanov's theorem for fBm]\label{thm:MFGirsanov}
	Let $u=\{u_t, t\in [0,T]\}$ be an $\Fbb^W$-adapted process with values in $\Hcal$ and integrable trajectories. If
	\begin{itemize}
		\item[(i)] $\int_0^{\cdot} u_s^{(k)} ds \in I_{0+}^{H_k+\frac{1}{2}} (L^2 [0,T])$, $\Pbb$-a.s. for every $k\geq 1$, and
		\item[(ii)] $\EW{\exp\left\lbrace \sum_{k\geq 1} \int_0^T \Kfrak_{H_k}^{-1} \left( \int_0^\cdot u_r^{(k)} dr \right)^2(s) ds \right\rbrace} < \infty$,	
	\end{itemize}
	where $\Kfrak_{H_k}^{-1}$ is defined as in \eqref{eq:isometrySingularInverse}, \eqref{eq:isometryRegularInverse}, and \eqref{eq:isometryBrownianInverse}, respectively, then the shifted process 
	\begin{align*}
		\widetilde{B}_t^\Hbb := B_t^\Hbb + \int_0^t u_s ds = \sum_{k\geq 1} \left( B_t^{H_k} + \int_0^t u_s^{(k)} ds \right) e_k,
	\end{align*}
	is a cylindrical fractional Brownian motion with associated sequence of Hurst parameters $\Hbb = \lbrace H_k \rbrace_{k\geq 1}$ under the new probability measure $\widetilde{\Pbb}$ defined by $\frac{d\widetilde{\Pbb}}{d\Pbb} := \Ecal_T$, where
	\begin{small}
	\begin{align}\label{eq:stochasticExponentialInfinite}
		\Ecal_T := \exp\left\lbrace \sum_{k\geq 1} \left( \int_0^T \Kfrak_{H_k}^{-1} \left( \int_0^\cdot u_r^{(k)} dr \right)(s) dW_s^{(k)} - \frac{1}{2} \int_0^T \Kfrak_{H_k}^{-1} \left( \int_0^\cdot u_r^{(k)} dr \right)^2(s) ds \right) \right\rbrace.
	\end{align}
	\end{small}
\end{theorem}

	It is shown in \cite{nualart2002regularization} that in the case $k \in I_- \cup I_0$ it is sufficient to assume $\int_0^T \vert u_s^{(k)} \vert^2  ds < \infty$ such that for $u^{(k)}$ condition (i) in \Cref{thm:MFGirsanov} is fulfilled. In the case $k \in I_+$ condition (i) in \Cref{thm:MFGirsanov} is fulfilled if the process $u^{(k)}$ is assumed to have Hölder continuous trajectories of order $H_k - \frac{1}{2} + \varepsilon$ for some $\varepsilon>0$. If we assume further that
	\begin{enumerate}[(i$^*$)]
	\setcounter{enumi}{1}
		\item $\int_0^T \Kfrak_{H_k}^{-1} \left( \int_0^\cdot u_r^{(k)} dr \right)^2 (s) ds \leq D_k$ $\Pbb$-a.s. for all $k\geq1$,
	\end{enumerate}
	where $D = \lbrace D_k \rbrace_{k\geq 1} \in \ell^1$ is a sequence of constants, then assumption (ii) is also fulfilled and thus Girsanov's theorem is applicable. We summarize these observations in the following corollary.

\begin{corollary}\label{cor:conditionsGirsanov}
	Let $(u_t)_{t\in[0,T]}$ be an $\Fbb^W$-adapted process such that $\int_0^T \vert u_s^{(k)} \vert^2  ds < \infty$ for all $k\in I_- \cup I_0$, and for $k \in I_+$ the process $u^{(k)}$ has Hölder continuous trajectories of order $H_k - \frac{1}{2} +  \varepsilon$ for some $\varepsilon>0$. Furthermore, assume that condition $(ii^\ast)$ is fulfilled. Then, conditions (i) and (ii) in \Cref{thm:MFGirsanov} are satisfied, and thus the stochastic exponential \eqref{eq:stochasticExponentialInfinite} defines the Radon-Nikodym density of a probability measure. Moreover, for every $p\in [0,\infty)$
	\begin{align*}
		\EpO{\Ecal_T}{p} < \infty.
	\end{align*}
\end{corollary}

\section{Existence and Uniqueness of Weak Solutions}\label{sec:solution}

	In this section we proof under sufficient conditions on the drift function $b$ the existence and uniqueness of weak solutions to the MKV equation \eqref{eq:mfsde}, where the weighted cylindrical fractional Brownian motion is characterized by a given sequence of Hurst parameters $\Hbb$ and the weighting operator $Q$. We show first existence of a weak solution using \Cref{thm:MFGirsanov} and Schauder's fixed point theorem. Afterwards weak uniqueness of the solution is proven. Let us first recall the definition of a weak solution and uniqueness in law, and then state the main result of this section.

\begin{definition}\label{def:MFweakSolution}
		We say the six-tuple $(\Omega, \Fcal, \Fbb, \Pbb, \Bb, X)$ is a \emph{weak solution} of MKV equation \eqref{eq:mfsde}, if
		\begin{enumerate}[(i)]
			\item $(\Omega, \Fcal, \Pbb)$ is a complete probability space and $\Fbb :=\lbrace \Fcal_t \rbrace_{t\in[0,T]}$ is a filtration on $(\Omega, \Fcal, \Pbb)$ satisfying the usual conditions of right-continuity and completeness,
			\item $X = (X_t)_{t\in[0,T]}$ is a continuous, $\Fbb$-adapted, $\Hcal$-valued process; $\Bb := (\Bb_t)_{t\in[0,T]}$ is a weighted cylindrical fractional Brownian motion with respect to $(\Fbb, \Pbb)$,
			\item $X$ satisfies $\Pbb$-a.s. MKV equation \eqref{eq:mfsde}, where $\Pbb_{X_t} \in \Pcal_1(\Hcal)$ denotes for all $t\in [0,T]$ the law of $X_t$ with respect to $\Pbb$.
		\end{enumerate}
\end{definition}

\begin{remark}
	We merely say that $X$ is a weak solution of MKV equation \eqref{eq:mfsde}, if there is no ambiguity about the filtered stochastic basis $(\Omega, \Fcal, \Fbb, \Pbb, \Bb)$.
\end{remark}	

\begin{definition}\label{def:MFweakUniqueness}
	A weak solution $(\Omega^1, \Fcal^1,\Fbb^1, \Pbb^1, \Bb^1 , X^1)$ of MKV equation \eqref{eq:mfsde} is called \emph{unique in law}, if for any other weak solution $(\Omega^2, \Fcal^2,\Fbb^2, \Pbb^2, \Bb^2, X^2)$ of \eqref{eq:mfsde} it holds that $\Pbb^1_{X^1} = \Pbb^2_{X^2}$, whenever $\Pbb^1_{X_0^1} = \Pbb^2_{X_0^2}$.
\end{definition}

\begin{theorem}\label{thm:weakSolution}
		Let $b: [0,T] \times \Hcal \times \Pcal_1(\Hcal) \to \Hcal$ be a measurable function such that $\Vert b_k \Vert_{\infty} \leq C_k \lambda_k$ for all $k\geq 1$, where $\frac{C}{\sqrt{1-\Hbb}} \in \ell^1$ for $C := \lbrace C_k \rbrace_{k\geq 1}$ and assume that
		\begin{align*}
			\left( \sum_{k\geq 1} \lambda_k^2 (t-s)^{2H_k} \right)^{\frac{1}{2}} \leq \rho \vert t-s \vert^\kappa,
		\end{align*}
		where $\rho >0$  and $0 < \kappa < 1$ are constants. Furthermore, assume that in the case $k \in I_+$,
	\begin{align}\label{eq:HoelderAssumptionb}
		\vert b_k (t,x,\mu) - b_k (s,y,\nu) \vert \leq C_k \lambda_k \left( \vert t-s \vert^{\gamma_k} + \Vert x-y \Vert_\Hcal^{\alpha_k} + \Kcal(\mu, \nu)^{\beta_k} \right),
	\end{align}
	where $\gamma_k > H_k - \frac{1}{2}$, $2 \geq \kappa \alpha_k > 2H_k - 1$, and $\kappa \beta_k > H_k - \frac{1}{2}$, and in the case $k \in I_- \cup I_0$ that for every $\mu \in \Ccal([0,T];\Pcal_1(\Hcal))$ and every $\varepsilon>0$ there exists $\delta>0$ such that for all $k\geq 1$ and $\nu \in \Ccal([0,T];\Pcal_1(\Hcal))$
		\begin{align}\label{eq:continuousUniformly}
			\sup_{t\in [0,T]} \Kcal(\mu_t,\nu_t) < \delta ~\Rightarrow \sup_{t\in[0,T], ~ y\in \Hcal}\left\vert b_k (t,y,\mu_t)-b_k (t,y,\nu_t) \right\vert < \varepsilon C_k \lambda_k.
		\end{align}
		Then, MKV equation \eqref{eq:mfsde} has a weak solution.
\end{theorem}

	The proof of \Cref{thm:weakSolution} is divided into two main steps. First we show using \Cref{thm:MFGirsanov} that for every $\mu \in \Ccal^\kappa([0,T]; \Pcal_1(\Hcal))$, for some suitable $\kappa > 0$, the (distribution dependent) SDE
	\begin{align}\label{eq:auxiliarySDE}
		dX_t^\mu = b\left( t, X_t^\mu, \mu_t \right) dt + d\Bb_t, \quad t\in[0,T], \quad X_0^\mu = x,
	\end{align}
	has a weak solution. Second, we apply Schauder's fixed point theorem, see \cite{Schauder}, to find a solution of MKV equation \eqref{eq:mfsde}. Let us start with the application of Girsanov's theorem in the following lemma.
	
\begin{lemma}\label{lem:solutionAuxiliarySDE}
	Let $b: [0,T] \times \Hcal \times \Pcal_1(\Hcal) \to \Hcal$ be a measurable function such that $\Vert b_k \Vert_{\infty} \leq C_k \lambda_k$ for all $k\geq 1$, where $\frac{C}{\sqrt{1-\Hbb}} \in \ell^1$. Furthermore, assume that for every $k \in I_+$ the function $b_k$ fulfills assumption \eqref{eq:HoelderAssumptionb}. Then for every $\mu \in \Ccal^{\kappa}([0,T]; \Pcal_1(\Hcal))$, SDE \eqref{eq:auxiliarySDE} has a weak solution which is unique in law.
\end{lemma}
\begin{proof}
	Let $(\Omega,\Fcal,\Fbb,\Pbb)$ be a complete filtered probability space with a sequence of independent Brownian motions $\lbrace W^{(k)} \rbrace_{k\geq 1}$ defined thereon. Following the constructions in \Cref{sec:wcfBm}, we define the cylindrical fractional Brownian motion $B^\Hbb$ with associated sequence of Hurst parameters $\Hbb = \lbrace H_k \rbrace_{k\geq 1}$ generated by $W$. Further, we define the process $X_t^\mu := x + \sqrt{Q} B_t^\Hbb$, $t\in [0,T]$. If $u_t := \sqrt{Q}^{-1} b(t,X_t^\mu, \mu_t)$, $t\in[0,T]$, fulfills the assumptions of \Cref{cor:conditionsGirsanov}, we get due to \Cref{thm:MFGirsanov} that the process
	\begin{align*}
		B_t^{\Hbb,\mu} := B_t^\Hbb - \int_0^t \sqrt{Q}^{-1} b\left(u,x+B_u^\Hbb,\mu_u\right) du, \quad t\in[0,T],
	\end{align*}
	is a cylindrical fractional Brownian motion with respect to the probability measure $\Pbb^{\mu}$ defined by $\frac{d\Pbb^{\mu}}{d\Pbb} := \Ecal_T^{\mu}$, where 
	\begin{small}
	\begin{align}\label{eq:densityMu}
			\Ecal_T^{\mu} := \exp\left\lbrace \sum_{k\geq 1} \left( \int_0^T \Kfrak_{H_k}^{-1} \left( \int_0^\cdot u_r^{(k)} dr \right)(s) dW_s^{(k)} - \frac{1}{2} \int_0^T \Kfrak_{H_k}^{-1} \left( \int_0^\cdot u_r^{(k)} dr \right)^2(s) ds \right) \right\rbrace.
		\end{align}
		\end{small}
	Consequently, the sextuple ${(\Omega, \Fcal,\Fbb,\Pbb^{\mu}, \sqrt{Q} B^{\Hbb,\mu}, X^{\mu})}$ is a weak solution of SDE \eqref{eq:auxiliarySDE}. Thus it is left to show that $u$ fulfills the assumptions of \Cref{cor:conditionsGirsanov}. \par
	Let $k\in I_- \cup I_0$. Then, 
	\begin{align*}
		\int_0^T \vert u_s^{(k)} \vert^2 ds &= \int_0^T \vert \lambda_k^{-1} b_k(s, X_s^\mu, \mu_s) \vert^2 ds \leq T C_k^2 < \infty,
	\end{align*}
	where we have used that $b_k$ is bounded by $\lambda_k C_k$. Consider now the case $k\in I_+$, then we get for $t,s \in [0,T]$ that
	\begin{align}\label{eq:bHoelder}
	\begin{split}
		\Eabs{u_t^{(k)} - u_s^{(k)}} &= \lambda_k^{-1} \Eabs{ b_k (t, X_t^\mu, \mu_t) - b_k (s, X_s^\mu, \mu_s)} \\
		&\leq C_k \left( \vert t-s \vert^{\gamma_k} + \EW{\left\Vert \sqrt{Q}B_t^\Hbb - \sqrt{Q}B_s^\Hbb \right\Vert_\Hcal^{\alpha_k}} + \Kcal(\mu_t, \mu_s)^{\beta_k} \right) \\
		&\leq C_k \left( \vert t-s \vert^{\gamma_k} + \left(\sum_{j\geq 1} \EW{\lambda_j^2 \left\vert B_t^{H_j} - B_s^{H_j} \right\vert^2}\right)^{\frac{\alpha_k}{2}} + \vert t-s \vert^{\kappa \beta_k} \right) \\
		&\leq C_k \left( \vert t-s \vert^{\gamma_k} + \left( \sum_{j\geq 1}\lambda_j^2 \vert t - s\vert^{2H_j} \right)^{\frac{\alpha_k}{2}}  \right) \\
		&\lesssim C_k \left( \vert t-s \vert^{\gamma_k} + \vert t-s \vert^{\frac{\kappa \alpha_k}{2}} \right) \lesssim \vert t-s \vert^{\gamma_k} + \vert t-s \vert^{\frac{\kappa \alpha_k}{2}},
	\end{split}
	\end{align}
	where we have assumed without loss of generality that $\gamma_k = \kappa \beta_k$. Due to Kolmogorov's continuity theorem and the assumptions $\gamma_k > H_k - \frac{1}{2}$ and $2 \geq \kappa \alpha_k > 2 H_k - 1$, we get that $u^{(k)}$ is $(H_k - \frac{1}{2} + \varepsilon)$--Hölder continuous in $t\in[0,T]$ for some $\varepsilon>0$ and hence, assumption (i) of \Cref{thm:MFGirsanov} is fulfilled for all $k\geq 1$ due to \Cref{cor:conditionsGirsanov}. Next, we show that assumption (ii$^\ast$) holds, i.e. for all $k\geq 1$ 
	\begin{align*}
		\int_0^T \Kfrak_{H_k}^{-1} \left( \int_0^\cdot u_r^{(k)} dr \right)^2 (s) ds \leq D_k,
	\end{align*}
	where $D = \lbrace D_k \rbrace_{k\geq 1} \in \ell^1$. Consider first the case $k \in I_0$, then
	\begin{align*}
		\int_0^T \Kfrak_{H_k}^{-1} \left( \int_0^\cdot u_r^{(k)} dr \right)^2 (s) ds = \int_0^T \vert \lambda_k^{-1} b_k (s, X_s^\mu, \mu_s) \vert^2 ds \leq T C_k^2,
	\end{align*}
	and thus we define $D_k := T C_k^2$ for $k\in I_0$. In the case $k \in I_-$ it is shown in \cite{banos2019infiniteSDE} that
	\begin{align*}
		\int_0^T \Kfrak_{H_k}^{-1} \left( \int_0^\cdot u_r^{(k)} dr \right)^2 (s) ds \lesssim T^2 C_k^2,
	\end{align*}
	and thus we define $D_k := T^2 C_k^2$ for $k\in I_-$. Last, we consider the case $k \in I_+$ and get that
	\begin{small}
	\begin{align}\label{eq:calcRegularCase}
		&\left\vert \Kfrak_{H_k}^{-1}\left( \int_0^\cdot u_r^{(k)} dr \right)(s)\right\vert = \left\vert \Kfrak_{H_k}^{-1}\left( \int_0^\cdot \lambda_k^{-1} b_k (r, X_r^\mu, \mu_r) dr \right)(s)\right\vert \notag \\
		&\quad \leq \frac{C_k s^{\frac{1}{2}-H_k}}{\Gamma\left(\frac{3}{2}-H_k\right)} + \frac{\left( H_k - \frac{1}{2} \right) s^{H_k-\frac{1}{2}}}{\lambda_k \Gamma\left(\frac{3}{2}-H_k\right)} \left\vert \int_0^s \frac{b_k (s, X_s^\mu, \mu_s) s^{\frac{1}{2}-H_k} - b_k (r, X_r^\mu, \mu_r) r^{\frac{1}{2}-H_k}}{(s-r)^{H_k+\frac{1}{2}}} dr \right\vert \notag \\
		&\quad \leq \frac{C_k s^{\frac{1}{2}-H_k}}{\Gamma\left(\frac{3}{2}-H_k\right)} + \frac{\left( H_k - \frac{1}{2} \right) s^{H_k-\frac{1}{2}}}{\lambda_k \Gamma\left(\frac{3}{2}-H_k\right)} \left( \int_0^s \left\vert b_k (s, X_s^\mu, \mu_s) \right\vert \frac{r^{\frac{1}{2}-H_k} - s^{\frac{1}{2}-H_k}}{(s-r)^{H_k+\frac{1}{2}}} dr \right. \notag \\
		&\quad \left. + \int_0^s r^{\frac{1}{2}-H_k} \frac{\left\vert b_k (s, X_s^\mu, \mu_s) - b_k (r, X_r^\mu, \mu_r) \right\vert}{(s-r)^{H_k+\frac{1}{2}}} dr \right).
	\end{align}
	\end{small}
	Due to \eqref{eq:bHoelder} there exists $\varepsilon > 0$ such that for all $k \in I_+$
	\begin{align*}
		\left\vert b_k (s, X_s^\mu, \mu_s) - b_k (r, X_r^\mu, \mu_r) \right\vert \lesssim C_k \lambda_k \vert s-r \vert^{H_k - \frac{1}{2} + \varepsilon}.
	\end{align*}
	Thus, \eqref{eq:calcRegularCase} can be further bounded by
	\begin{align*}
		&\left\vert \Kfrak_{H_k}^{-1}\left( \int_0^\cdot u_r^{(k)} dr \right)(s)\right\vert \lesssim \frac{C_k s^{\frac{1}{2}-H_k}}{\Gamma\left(\frac{3}{2}-H_k\right)} + \frac{C_k \left( H_k - \frac{1}{2} \right) s^{H_k-\frac{1}{2}}}{\Gamma\left(\frac{3}{2}-H_k\right)} \notag \\
		&\qquad \times \left( \int_0^s \frac{r^{\frac{1}{2}-H_k} - s^{\frac{1}{2}-H_k}}{(s-r)^{H_k+\frac{1}{2}}} dr + \int_0^s r^{\frac{1}{2}-H_k} ( s-r )^{\varepsilon - 1} dr \right) \notag \\
		&\quad \leq \frac{C_k s^{\frac{1}{2}-H_k}}{\Gamma\left(\frac{3}{2}-H_k\right)} + \frac{C_k \left( H_k - \frac{1}{2} \right) s^{H_k-\frac{1}{2}}}{\Gamma\left(\frac{3}{2}-H_k\right)} \notag \\
		&\qquad \times \left( s^{1-2H_k} \int_0^1 \frac{u^{\frac{1}{2}-H_k} - 1}{(1-u)^{\frac{1}{2}+H_k}} du+ s^{\frac{1}{2} - H_k + \varepsilon} \beta\left( \frac{3}{2} - H_k, \varepsilon \right) \right) \notag \\
		&\quad \leq \frac{C_k s^{\frac{1}{2}-H_k}}{\Gamma\left(\frac{3}{2}-H_k\right)} + \frac{C_k \left( H_k - \frac{1}{2} \right) s^{\frac{1}{2}- H_k}}{\Gamma\left(\frac{3}{2}-H_k\right)} + \frac{C_k \left( H_k - \frac{1}{2} \right) s^{\varepsilon}}{\Gamma\left(\frac{3}{2}-H_k\right)} \beta\left( \frac{3}{2} - H_k, \varepsilon \right) \notag \\
		&\quad \lesssim C_k s^{\frac{1}{2}- H_k} + C_k.
	\end{align*}
	Here, we have used that
	\begin{align*}
		\sup_{\alpha \in \left(0,\frac{1}{2}\right)} \int_0^1 \frac{u^{-\alpha} - 1}{(1-u)^{\alpha+1}} du < \infty.
	\end{align*}
	Integrating the squared of the inverse kernel over the time interval $[0,T]$ yields
	\begin{align*}
		\int_0^T \left\vert \Kfrak_{H_k}^{-1}\left( \int_0^\cdot u_r^{(k)} dr \right)(s)\right\vert^2 ds \leq 2 C_k^2 \left( \int_0^T s^{1-2H_k} ds + 1 \right) \lesssim \frac{1}{1 - H_k} C_k^2,
	\end{align*}
	and thus we define $D_k := \frac{C_k^2}{1 - H_k}$ for $k\in I_+$. Finally, we see that $D \in \ell^1$. Indeed,
	\begin{align*}
		\sum_{k\geq 1} D_k = T \sum_{k\in I_0} C_k^2 + T^2 \sum_{k\in I_-} C_k^2 + \sum_{k\in I_+} \frac{C_k^2}{1 - H_k} \lesssim \sum_{k\geq 1} \frac{C_k^2}{1-H_k},
	\end{align*}
	which is finite by assumption. Thus the stochastic exponential $\Ecal_T^{\mu}$ is well-defined and gives the probability measure $\Pbb^\mu$. If $\Ecal_T^{\mu}$ is invertible, the solution of SDE \eqref{eq:auxiliarySDE} is unique in law. Indeed, let $X$ and $Y$ be two solutions of SDE\eqref{eq:auxiliarySDE} with respect to the measures $\Pbb$ and $\Qbb$, respectively. Then, we have for every bounded functional $f: \Hcal \to \Rbb$ that
	\begin{align*}
		\Ebb_{\Pbb}[f(X)] = \Ebb_{\Pbb^\mu}\left[f\left(x+\sqrt{Q} B^{\Hbb,\mu}\right)\eta_T\right] = \Ebb_{\Qbb}[f(Y)],
	\end{align*}
	and thus $X$ and $Y$ have the same law. Here,
	\begin{small}
	\begin{align*}
		\eta_T := \exp\left\lbrace \sum_{k\geq 1} \left( -\int_0^T \Kfrak_{H_k}^{-1} \left( \int_0^\cdot u_r^{(k)} dr \right)(s) d\widetilde{W}_s^{(k)} - \frac{1}{2} \int_0^T \Kfrak_{H_k}^{-1} \left( \int_0^\cdot u_r^{(k)} dr \right)^2(s) ds \right) \right\rbrace,
	\end{align*}
	\end{small}
is the inverse of $\Ecal_T^\mu$, where $\widetilde{W} = \lbrace \widetilde{W}^{(k)} \rbrace_{k\geq 1}$ is a sequence of independent Brownian motions with respect to the measure $\Pbb^\mu$ which generate the fractional Brownian motions $\lbrace B^{H_k, \mu} \rbrace_{k\geq 1}$. \par
	 In order to show that $\eta_T$ is well-defined it suffices by \Cref{cor:conditionsGirsanov} to prove that the assumptions (i) and (ii$^*$) are fulfilled. Due to the proof of the existence of a weak solution of SDE \eqref{eq:auxiliarySDE}, in particular the derivation in \eqref{eq:bHoelder}, it suffices to show that for every $k\in I_+$
	\begin{align*}
		\EW{ \vert X_t^{(k), \mu} - X_s^{(k), \mu} \vert^2} \lesssim \vert t-s \vert^{2H_k}.
	\end{align*}
	Using Hölder's inequality and the fact that $X^\mu$ solves the SDE \eqref{eq:auxiliarySDE} we get for every $k\in I_+$ that
	\begin{align*}
		\Ebb_{\Pbb^\mu} \left[\vert X_t^{(k),\mu} - X_s^{(k),\mu} \vert^2\right] &= \Ebb_{\Pbb^\mu} \left[ \left\vert \int_s^t b_k (r,X_r^\mu, \mu_r) dr + \lambda_k B_t^{H_k, \mu} - \lambda_k B_s^{H_k, \mu} \right\vert^2 \right] \\
		&\lesssim \left( C_k^2 \lambda_k^2 \vert t-s \vert^2 + \lambda_k^2 \vert t-s \vert^{2H_k} \right) \lesssim \vert t-s \vert^{2H_k}.
	\end{align*}
	Consequently, $\Ecal_T^\mu$ is invertible and thus the solution is unique in law.
\end{proof}

As a direct consequence of the proof of \Cref{lem:solutionAuxiliarySDE} we get under the assumption that there are no Hurst parameters of the regular case, i.e. $I_+ = \emptyset$, existence and uniqueness (in law) of a solution for an even broader class of drift coefficients $b$ and measures $\mu$.

\bigskip

\begin{corollary}\label{cor:solutionAuxiliarySDE}
	Assume  $I_+ = \emptyset$. Let $b: [0,T] \times \Hcal \times \Pcal_1(\Hcal) \to \Hcal$ be a measurable function such that $\Vert b_k \Vert_{\infty} \leq C_k \lambda_k$ for all $k\geq 1$, where $C \in \ell^1$. Then SDE \eqref{eq:auxiliarySDE} has a weak solution which is unique in law for every $\mu \in \Ccal([0,T]; \Pcal_1(\Hcal))$.
\end{corollary}

\bigskip

Next, we come to the second step of the proof of \Cref{thm:weakSolution}, namely the application of Schauder's fixed point theorem, see \cite{Schauder}.

\begin{proof}[Proof of \Cref{thm:weakSolution}]
		Define $E := \Ccal^\kappa([0,T];\Pcal_1(\Hcal)) \subset \Ccal([0,T];\Mcal_1(\Hcal))$. Then \Cref{lem:solutionAuxiliarySDE} yields that SDE \eqref{eq:auxiliarySDE} has a weak solution $X^\mu$ which is unique in law for every $\mu \in E$. \par 
		Consider the function $\psi: E \to \Ccal([0,T];\Mcal_1(\Hcal))$ defined by
		\begin{align*}
			\psi_s(\mu) := \Pbb_{X_s^{\mu}}^{\mu}, \quad s\in [0,T].
		\end{align*}
		If $\psi$ has a fixed point, i.e. $\mu_s^* = \psi_s(\mu^*) = \Pbb_{X_s^{\mu^*}}^{\mu^*}$, $s\in [0,T]$, we can insert $\mu^*$ in SDE \eqref{eq:auxiliarySDE} and  consequently get a weak solution of MKV equation \eqref{eq:mfsde}. In order to apply Schauder's fixed point theorem we have to verify that $(E,\Vert \cdot \Vert_{\Kcal^*})$ is convex, $\psi$ is continuous, and it exists a compact subset $G$ of $E$ such that $\psi(E) \subset G \subset E$. \par
	\emph{$(E,\Vert \cdot \Vert_{\Kcal^*})$ is convex.} This is an immediate consequence of the definition of $E$ and the fact that the Kantorovich-Rubinstein metric $\Kcal$ is induced by the Kantorovich norm $\Vert \cdot \Vert_{\Kcal}$. \par 
	\emph{$\psi$ is continuous.} Consider an arbitrary $\mu \in E$ and let $\varepsilon > 0$. Due to the continuity assumption \eqref{eq:continuousUniformly} on $b$, we can find $\delta>0$ such that for every $\nu \in E$ with $\sup_{t\in [0,T]}\Kcal(\mu_t,\nu_t) < \delta$
		\begin{align*}
			\sup_{t\in[0,T], y\in \Hcal} \left\vert b_k (t,y,\mu_t) - b_k (t,y,\nu_t) \right\vert < C_k \lambda_k \varepsilon, \quad k\geq 1.
		\end{align*}
		Consequently, we get by the measure change defined in \eqref{eq:densityMu} and Cauchy-Schwarz' inequality that
		\begin{align*}
			\Kcal&(\psi_t(\mu),\psi_t(\nu)) = \sup_{h \in \BL(\Hcal;\Rbb)} \left\vert \int_{\Hcal} h(y) \Pbb_{X_t^{\mu}}^{\mu}(dy) - \int_{\Hcal} h(y) \Pbb_{X_t^{\nu}}^{\nu}(dy) \right\vert \\
			&= \sup_{h\in \BL(\Hcal;\Rbb)} \left\vert \Ebb \left[ \left(h\left(\Bb_t^x \right) - h(x)\right)\Ecal_T^{\mu} \right] - \Ebb \left[ \left(h\left(\Bb_t^x \right) - h(x)\right)\Ecal_T^{\nu} \right] \right\vert\\
			&\leq \Ebb \left[ \left\Vert \Bb_t \right\Vert_{\Hcal} \left\vert \Ecal_T^{\mu} - \Ecal_T^{\nu} \right\vert \right] \leq \Ebb \left[ \left\Vert \Bb_t \right\Vert_{\Hcal}^2 \right]^{\frac{1}{2}} \Ebb\left[\left\vert \Ecal_T^{\mu} - \Ecal_T^{\nu} \right\vert^2 \right]^{\frac{1}{2}}.
		\end{align*}
		Note that $\sup_{t \in [0,T]}\Ebb \left[ \left\Vert \Bb_t \right\Vert_{\Hcal}^2 \right]$ is finite due to \Cref{lem:continuity&L2}. We now employ the inequality
		\begin{align}\label{eq:exponentialInequality}
			\left\vert e^x - e^y \right\vert \leq \vert x-y \vert \left(e^x + e^y\right), \quad x,y \in \Rbb.
		\end{align}
		Since $\Ecal_T^{\mu} \in L^p(\Omega)$ for every $\mu \in E$ and $1 \leq p < \infty$ by \Cref{lem:continuity&L2}, we get again by Cauchy-Schwarz' and Minkowski's inequality that
		\begin{align*}
			 &\Ebb\left[\vert \Ecal_T^{\mu} - \Ecal_T^{\nu} \vert^2 \right]^{\frac{1}{2}} \lesssim \Ebb \left[ \left\vert \sum_{k\geq 1} \int_0^T \lambda_k^{-1} \left( \Kfrak_{H_k}^{-1} \left( \int_0^{\cdot} b_k \left(u,\Bb_u^x, \mu_u\right) du \right)(s) \right.\right.\right. \\
			 &\quad \left.\left.\left.- \Kfrak_{H_k}^{-1} \left( \int_0^{\cdot} b_k \left(u,\Bb_u^x, \nu_u\right) du \right)(s)  \right) dW_s^{(k)}\right\vert^4 \right]^{\frac{1}{4}} \\
			 &\quad + \frac{1}{2} \Ebb \left[ \left\vert \sum_{k\geq 1} \int_0^T \lambda_k^{-2} \left( \Kfrak_{H_k}^{-1} \left( \int_0^{\cdot} b_k \left(u,\Bb_u^x, \mu_u\right) du \right)^2(s) \right.\right.\right.\\
			 &\quad \left.\left.\left.- \Kfrak_{H_k}^{-1} \left( \int_0^{\cdot} b_k \left(u,\Bb_u^x, \nu_u\right) du \right)^2(s)  \right) ds \right\vert^4 \right]^{\frac{1}{4}} =: A + B.
		\end{align*}
		For $A$ we get equivalently to \Cref{lem:solutionAuxiliarySDE} using the linearity of $\Kfrak_H^{-1}$ for every $H\in (0,1)$ and Burkholder-Davis-Gundy's inequality that
		\begin{align*}
			 A &\lesssim \Ebb \left[ \sum_{k\geq 1} \left( \int_0^T \frac{1}{\lambda_k^{2}} \Kfrak_{H_k}^{-1} \left( \int_0^{\cdot} b_k \left(u,\Bb_u^x, \mu_u\right) - b_k \left(u,\Bb_u^x, \nu_u\right) du \right)^2(s)  ds \right)^2 \right]^{\frac{1}{4}} \\
			 &\lesssim  \left( \sum_{k\geq 1} D_k \varepsilon^2 \right)^{\frac{1}{2}} \lesssim \varepsilon.
		\end{align*}
		For $B$ note that 
		\begin{align*}
			B &\lesssim \Ebb \left[ \left\vert \sum_{k\geq 1} \frac{1}{\lambda_k^2}\int_0^T \left( \Kfrak_{H_k}^{-1} \left( \int_0^{\cdot} b_k \left(u,\Bb_u^x, \mu_u\right) + b_k \left(u,\Bb_u^x, \nu_u\right) du \right)(s)  \right) \right. \right.\\
			&\quad \times \left. \left. \left( \Kfrak_{H_k}^{-1} \left( \int_0^{\cdot} b_k \left(u,\Bb_u^x, \mu_u\right) - b_k \left(u,\Bb_u^x, \nu_u\right) du \right)(s)  \right)ds \right\vert^4 \right]^{\frac{1}{4}},
		\end{align*}
		which can be bounded equivalently to $A$. Hence, $\psi$ is continuous. \par
	\emph{$\psi$ maps $E$ onto itself.} It suffices to show that  for every $\mu \in E$
		\begin{align*}
			\Kcal(\psi_t(\mu), \psi_s(\mu)) \lesssim \vert t-s \vert^\kappa.
		\end{align*}
		Let $\mu \in E$ be arbitrary and without loss of generality $s<t$. Then we get
		\begin{align*}
			\Kcal(\psi_t(\mu), \psi_s(\mu)) &= \sup_{h\in \BL(\Hcal;\Rbb)} \left\vert \Ebb\left[ h(X_t^{\mu})-h(X_s^{\mu}) \right] \right\vert \leq \Ebb \left[ \left\Vert X_t^{\mu} - X_s^{\mu} \right\Vert_{\Hcal}^2 \right]^{\frac{1}{2}}\notag \\
			&= \Ebb \left[ \left\Vert \int_s^t b\left(u,X_u^{\mu},\mu_u\right) du + \Bb_t - \Bb_s \right\Vert_{\Hcal}^2 \right]^{\frac{1}{2}} \notag\\
			&\leq \left( \sum_{k\geq 1} C_k^2 \lambda_k^2 \right)^{\frac{1}{2}} (t-s) + \left( \sum_{k\geq 1} \lambda_k^2 \Ebb \left[ \left\vert B_t^{H_k,\mu}-B_s^{H_k,\mu} \right\vert^2 \right]\right)^{\frac{1}{2}} \notag\\
			&\leq \left( \sum_{k\geq 1} C_k^2 \lambda_k^2 \right)^{\frac{1}{2}} (t-s) + \left( \sum_{k\geq 1} \lambda_k^2 (t-s)^{2H_k} \right)^{\frac{1}{2}} \lesssim \vert t-s \vert^\kappa.
		\end{align*}	
	\emph{$\exists G \subset E$ compact such that $\psi(E) \subset G \subset E$.} Define
		\begin{align*}
			\Delta := \left\lbrace \Pbb_{X_s^{\mu}}^{\mu}, ~s\in [0,T], ~\mu \in \Ccal^\kappa([0,T];\Pcal_1(\Hcal)) \right\rbrace \subset \Pcal_1(\Hcal).
		\end{align*}
		By the last step, we already know that for $s,t \in [0,T]$ and $\mu \in \Ccal^\kappa([0,T]; \Pcal_1(\Hcal))$,
		\begin{align*}
			\Kcal(\Pbb_{X_t^{\mu}}^{\mu}, \Pbb_{X_s^{\mu}}^{\mu}) \lesssim \vert t-s\vert^\kappa.
		\end{align*}
		Hence, $\psi(E) \subset G := \Ccal^\kappa([0,T]; \overline{\Delta}) \subset E$, where $\overline{\Delta}$ is the closure of $\Delta$ with respect to the Kantorovich-Rubinstein metric. If we can show that $\Delta$ is relatively compact, then $G$ will be compact. \par
		Indeed, note first that $G$ is a closed set of equicontinuous functions. Moreover, for every $s\in [0,T]$ the set
		\begin{align*}
			G_s := \left\lbrace \Pbb_{X_s^{\mu}}^{\mu}, ~\mu \in \Ccal^\kappa([0,T];\Pcal_1(\Hcal)) \right\rbrace \subset \overline{\Delta}
		\end{align*}
		is relatively compact due to the compactness of $\overline{\Delta}$. Hence, we can apply Arzelá-Ascoli's theorem which shows the compactness of $G$ with respect to the metric induced by $\Vert \cdot \Vert_{\Kcal^*}$.\par 
		In order to show relatively compactness of $\Delta$, note first that relatively compactness of $\Delta$ is equivalent to tightness of $\Delta$. Tightness of $\Delta$ then again is implied by uniformly integrability of the set
		\begin{align*}
			\mathcal{X} := \lbrace X_s^{\mu}, ~s\in[0,T], ~\mu\in \Ccal^\kappa([0,T];\Pcal_1(\Hcal)) \rbrace.
		\end{align*}
		Hence, it suffices to show that 
		\begin{align*}
			\sup_{s\in[0,T]} \sup_{\mu \in \Ccal^\kappa([0,T];\Pcal_1(\Hcal))} \Ebb \left[ \Vert X_s^{\mu} \Vert_{\Hcal}^2 \right] < \infty,
		\end{align*}
		but this follows directly due to \Cref{lem:continuity&L2} and the observation
		\begin{align*}
			\Ebb \left[ \Vert X_s^{\mu} \Vert_{\Hcal}^2 \right] &= \Ebb \left[ \Vert x + \int_0^s b(r, X_r^{\mu}, \mu_r) dr + \Bb_s \Vert_{\Hcal}^2 \right] \lesssim \Vert x \Vert_\Hcal^2 + T^2 \Vert C \lambda \Vert_{\ell^2} + \Vert \Bb_s \Vert_{\Hcal}^2.
		\end{align*}
		Finally, we can apply Schauder's fixed point theorem, which yields a fixed point $\mu^* = \psi(\mu^*) = \Pbb_{X^{\mu^*}}^{\mu^*}$. Define $\Pbb:= \Pbb^{\mu^*}$, $X:= X^{\mu^*}$ and $B^\Hbb := B^{\Hbb,\mu^*}$. Then, $(\Omega, \Fcal, \Fbb, \Pbb, B^\Hbb, X)$ is a weak solution of MKV equation \eqref{eq:mfsde}.
\end{proof}

For the case $I_+ = \emptyset$ we get an immediate extension of \Cref{thm:weakSolution}.

\begin{corollary}
	Assume $I_+ = \emptyset$. Let $b: [0,T] \times \Hcal \times \Pcal_1(\Hcal) \to \Hcal$ be a measurable function such that $\Vert b_k \Vert_{\infty} \leq C_k \lambda_k$ for all $k\geq 1$, where $C \in \ell^1$, and assume that $b$ is continuous in the sense of \eqref{eq:continuousUniformly}. Then, MKV equation \eqref{eq:mfsde} has a weak solution.
\end{corollary}
\begin{proof}
	The proof is analog to the proof of \Cref{thm:weakSolution}, where we define the sets
	\begin{footnotesize}
		\begin{align*}
			E:= \left\lbrace \mu \in \Ccal([0,T];\Pcal_1(\Hcal)) : \Kcal(\mu_t,\mu_s) \leq \left( \sum_{k\geq 1} C_k^2 \lambda_k^2 \right)^{\frac{1}{2}} (t-s) + \left( \sum_{k\geq 1} \lambda_k^2 (t-s)^{2H_k} \right)^{\frac{1}{2}} \right\rbrace,
		\end{align*}
		\end{footnotesize}
		and
		\begin{footnotesize}
		\begin{align*}
			G := \left\lbrace \mu \in \Ccal([0,T]; \overline{\Delta}) : \Kcal(\mu_t,\mu_s) \leq \left( \sum_{k\geq 1} C_k^2 \lambda_k^2 \right)^{\frac{1}{2}} (t-s) + \left( \sum_{k\geq 1} \lambda_k^2 (t-s)^{2H_k} \right)^{\frac{1}{2}} \right\rbrace.
		\end{align*}
		\end{footnotesize}
\end{proof}

Concluding this section we show that under slightly more regularity in the law variable of the drift $b$ we get a solution which is unique in law.

\begin{theorem}\label{thm:weakUniqueness}
	Suppose the assumptions of \Cref{thm:weakSolution} are fulfilled and in addition that $\sup_{k\in I_+} H_k < 1$. Furthermore, for every $k\geq 1$ assume that for all $\mu, \nu \in \Pcal_1(\Hcal)$
	\begin{align}\label{eq:lipschitzUniformly}
		\sup_{t\in[0,T], y\in \Hcal}\left\vert b_k (t,y,\mu)-b_k (t,y,\nu) \right\vert \leq C_k \lambda_k \Kcal(\mu,\nu).
	\end{align}
	Then, MKV equation \eqref{eq:mfsde} has a weak solution which is unique in law.
\end{theorem}
\begin{proof}
	In this proof we proceed similar to \cite[Theorem 2.7]{Bauer_StrongSolutionsOfMFSDEs}. Let $(\Omega, \Fcal,\Fbb, \Pbb, \Bb, X)$ and $(\widetilde{\Omega}, \widetilde{\Fcal},\widetilde{\Fbb}, \widetilde{\Pbb}, \widetilde{\Bb}, Y)$ be two weak solutions of MKV equation \eqref{eq:mfsde} such that $X_0 = Y_0 = x \in \Hcal$. For the sake of readability we assume $x$ to be the Null element in $\Hcal$ whereas the general case can be shown analogously. Furthermore we denote by $B^\Hbb$ and $\widetilde{B}^\Hbb$ the cylindrical fractional Brownian motions related to $\Bb$ and $\widetilde{\Bb}$, respectively. Lastly, we denote by $\lbrace W^{(k)} \rbrace_{k\geq 1}$ and $\lbrace \widetilde{W}^{(k)} \rbrace_{k\geq 1}$ the generating sequences of Brownian motions of $B^\Hbb$ and $\widetilde{B}^\Hbb$, respectively. \par
	Due to the proof of \Cref{thm:MFGirsanov} and \Cref{thm:weakSolution} there exist probability measures $\Qbb$ and $\widetilde{\Qbb}$ such that $X$ and $Y$ are weighted cylindrical fractional Brownian motions of the form \eqref{eq:qCylindricalFBm} under  $\Qbb$ and $\widetilde{\Qbb}$, respectively. Furthermore, we define the probability measure $\widehat{\Qbb} \approx \widetilde{\Pbb}$ by
	\begin{align*}
		\frac{d \widehat{\Qbb}}{d\widetilde{\Pbb}} &:= \exp\left\lbrace -\sum_{k\geq 1} \int_0^t \lambda_k^{-1} \Kfrak_{H_k}^{-1} \left( \int_0^{\cdot} b_k\left(u,Y_u, \widetilde{\Pbb}_{Y_u}\right) - b_k\left(u,Y_u,\Pbb_{X_u}\right) du \right) (s) d\widetilde{W}_s^{(k)} \right.\\
		&\quad \left. \quad -\frac{1}{2} \sum_{k\geq 1} \int_0^t \lambda_k^{-2} \Kfrak_{H_k}^{-1}\left(\int_0^{\cdot} b_k \left(u,Y_u, \widetilde{\Pbb}_{Y_u}\right) - b_k \left(u,Y_u,\Pbb_{X_u}\right) du \right)^2(s) ds\right\rbrace,
	\end{align*}
	and the $\widehat{\Qbb}$ cylindrical fractional Brownian motion
	\begin{align*}
		\widehat{B}_t^\Hbb := \widetilde{B}_t^\Hbb + \int_0^t \sqrt{Q}^{-1} \left( b \left( s, Y_s, \widetilde{\Pbb}_{Y_s} \right) - b\left( s, Y_s, \Pbb_{X_s} \right) \right) ds, \quad t\in [0,T].
	\end{align*}
	Note that we can find a measurable function $\Phi: [0,T] \times \Ccal([0,T]; \Hcal) \to \Hcal$ such that 
	\begin{align*}
		B_t^\Hbb = \Phi_t(X) \quad \text{ and } \quad \widehat{B}_t^\Hbb = \Phi_t(Y),
	\end{align*}
	since
	\begin{align*}
		B_t^\Hbb &= \sqrt{Q}^{-1} \left( X_t - \int_0^t b\left( s,X_s, \Pbb_{X_s} \right) ds \right), \text{ and}\\
		\widehat{B}_t^\Hbb &= \sqrt{Q}^{-1} \left( Y_t - \int_0^t b\left( s,Y_s, \Pbb_{X_s} \right) ds \right).
	\end{align*}
	Consequently,
	\begin{align*}
		\Ebb_{\Pbb} \left[F(B^\Hbb, X) \right] &= \Ebb_{\Qbb} \left[\Ecal\left(\int_0^T \sqrt{Q}^{-1} b\left(t, X_t, \Pbb_{X_t}\right) dX_t \right) F(\Phi(X), X) \right] \\
		&= \Ebb_{\widetilde{\Qbb}} \left[\Ecal\left(\int_0^T \sqrt{Q}^{-1} b\left(t, Y_t, \Pbb_{X_t}\right) dY_t \right) F(\Phi(Y), Y) \right] \\
		&= \Ebb_{\widehat{\Qbb}} \left[F(\widehat{B}^\Hbb, Y) \right],
	\end{align*}
		for every bounded measurable functional $F: \Ccal([0,T]; \Hcal) \times \Ccal([0,T]; \Hcal) \to \Rbb$ and thus $\Pbb_{(B^\Hbb, X)} = \widehat{\Qbb}_{(\widehat{B}^\Hbb, Y)}$. Therefore it is left to show that $\sup_{t \in [0,T]} \Kcal \left( \widehat{\Qbb}_{Y_t}, \widetilde{\Pbb}_{Y_t} \right) = 0$ from which we can conclude that $\frac{d\widehat{\Qbb}}{d \widetilde{\Pbb}}=1$ and in particular that $\Pbb_X = \widetilde{\Pbb}_Y$. \par
	Applying a measure change, inequality \eqref{eq:exponentialInequality}, Burkholder-Davis-Gundy's inequality, and assumption \eqref{eq:lipschitzUniformly}, yield
	\begin{align*}
		\Kcal&\left(\Pbb_{X_t},\widetilde{\Pbb}_{Y_t}\right) = \sup_{h\in\BL(\Hcal; \Rbb)} \left| \Ebb_{\widehat{\Qbb}} \left[h(Y_t)-h(0) \right] - \Ebb_{\widetilde{\Pbb}}\left[ h(Y_t)-h(0)\right] \right| \\
		&\leq \sup_{h\in\BL(\Hcal;\Rbb)} \Ebb_{\widetilde{\Pbb}} \left[\left|\frac{d \widehat{\Qbb}}{d\widetilde{\Pbb}} - 1\right| \left| h\left(Y_t\right)-h(0)\right|\right] \\
		&\leq \Ebb_{\widetilde{\Pbb}} \left[\left|\frac{d \widehat{\Qbb}}{d\widetilde{\Pbb}} - 1\right|^2\right]^{\frac{1}{2}} \Ebb_{\widetilde{\Qbb}}\left[ \left( \frac{d\widetilde{\Pbb}}{d \widetilde{\Qbb}} \right)^2\right]^{\frac{1}{4}} \Ebb_{\widetilde{\Qbb}}\left[ \left\Vert \widetilde{B}_t^\Hbb \right\Vert_{\Hcal}^{4}\right]^{\frac{1}{4}} \\
		& \lesssim \Ebb \left[\left| \sum_{k\geq 1} \int_0^t \lambda_k^{-2} \Kfrak_{H_k}^{-1} \left( \int_0^{\cdot} b_k \left(u,\Bb_u, \widetilde{\Pbb}_{Y_u}\right) - b_k \left(u,\Bb_u,\Pbb_{X_u}\right) du \right)^2 (s) ds \right\vert^2 \right]^{\frac{1}{4}}\\
		&\quad + \Ep{\sum_{k\geq 1} \int_0^t \lambda_k^{-2} \Kfrak_{H_k}^{-1}\left(\int_0^{\cdot} b_k \left(u,\Bb_u, \widetilde{\Pbb}_{Y_u}\right) - b_k \left(u,\Bb_u,\Pbb_{X_u}\right) du \right)^2(s) ds}{4} =:  A.
	\end{align*}
	Consider first the Brownian case $k \in I_0$. Then, we get
	\begin{align*}
		&\int_0^t \Kfrak_{H_k}^{-1} \left( \int_0^{\cdot} b_k \left(u,\Bb_u, \widetilde{\Pbb}_{Y_u}\right) - b_k \left(u,\Bb_u,\Pbb_{X_u}\right) du \right)^2 (s) ds \leq C_k^2 \lambda_k^2 \int_0^t \Kcal(\Pbb_{X_s}, \widetilde{\Pbb}_{Y_s})^2 ds.
	\end{align*}		
	In the singular case $k\in I_-$, we have
	\begin{align*}
		&\int_0^t \Kfrak_{H_k}^{-1} \left( \int_0^{\cdot} b_k \left(u,\Bb_u, \widetilde{\Pbb}_{Y_u}\right) - b_k \left(u,\Bb_u,\Pbb_{X_u}\right) du \right)^2 (s) ds\\
		&\quad \leq \frac{C_k^2\lambda_k^2}{\Gamma\left(\frac{1}{2}-H_k\right)^2} \int_0^t s^{2H_k-1} \Kcal(\Pbb_{X_s}, \widetilde{\Pbb}_{Y_s})^2 \left(\int_0^s (s-u)^{-H_k-\frac{1}{2}} u^{\frac{1}{2}-H_k} du \right)^2 ds \\
		&\quad \leq \frac{C_k^2 \lambda_k^2}{\Gamma\left(\frac{1}{2}-H_k\right)^2} \int_0^t s^{1-2H_k} \Kcal(\Pbb_{X_s}, \widetilde{\Pbb}_{Y_s})^2 \beta\left( \frac{3}{2} - H_k, \frac{1}{2} - H_k \right)^2 ds \\
		&\quad \leq \frac{C_k^2 \lambda_k^2 T^{1-2H_k} \Gamma\left(\frac{3}{2}- H_k \right)^2}{\Gamma\left(2-2H_k\right)^2} \int_0^t \Kcal(\Pbb_{X_s}, \widetilde{\Pbb}_{Y_s})^2  ds \\
		&\quad \lesssim C_k^2 \lambda_k^2 \int_0^t \Kcal(\Pbb_{X_s}, \widetilde{\Pbb}_{Y_s})^2  ds.
	\end{align*}
	Lastly we get in the regular case $k \in I_+$ equivalent to \eqref{eq:calcRegularCase} that
	\begin{align*}
		&\int_0^t \Kfrak_{H_k}^{-1} \left( \int_0^{\cdot} b_k \left(u,\Bb_u, \widetilde{\Pbb}_{Y_u}\right) - b_k \left(u,\Bb_u,\Pbb_{X_u}\right) du \right)^2 (s) ds  \\
		&\quad \lesssim C_k^2 \lambda_k^2 \int_0^t \Kcal\left(\Pbb_{X_s}, \widetilde{\Pbb}_{Y_s} \right)^2 s^{1-2H_k} ds.
	\end{align*}
	
	Using Hölder's inequality with $1<p< \frac{1}{2 \sup_{k\in I_+} H_k -1}$ and its conjugate $q>1$ yields
	
	\begin{align*}
		&\int_0^t \Kcal\left(\Pbb_{X_s}, \widetilde{\Pbb}_{Y_s} \right)^2 s^{1-2H_k} ds \\
		&\quad \leq \left( \int_0^t \Kcal\left(\Pbb_{X_s}, \widetilde{\Pbb}_{Y_s} \right)^{2q} ds \right)^{\frac{1}{q}} \left( \int_0^t  s^{p(1-2H_k)} ds \right)^{\frac{1}{p}} \\
		&\quad \leq \left( \int_0^t \Kcal\left(\Pbb_{X_s}, \widetilde{\Pbb}_{Y_s} \right)^{2q} ds \right)^{\frac{1}{q}} \left( \frac{1}{p(1-2 H_k) + 1} t^{p(1-2H_k) + 1} \right)^{\frac{1}{p}} \\
		&\quad \lesssim \left( \int_0^t \Kcal\left(\Pbb_{X_s}, \widetilde{\Pbb}_{Y_s} \right)^{2q} ds \right)^{\frac{1}{q}}.
	\end{align*}

	Consequently,
	\begin{small}
	\begin{align*}
		\Kcal(\Pbb_{X_t}, \widetilde{\Pbb}_{Y_t}) &\lesssim \left( \sum_{k\geq 1} C_k^2 \left( \int_0^t \Kcal \left(\Pbb_{X_s}, \widetilde{\Pbb}_{Y_s}\right)^{2q} ds \right)^{\frac{1}{q}} \right)^{\frac{1}{2}} + \sum_{k\geq 1} C_k^2 \left( \int_0^t \Kcal\left(\Pbb_{X_s}, \widetilde{\Pbb}_{Y_s}\right)^{2q} ds \right)^{\frac{1}{q}}\\
		&\lesssim \left( \int_0^t \Kcal\left(\Pbb_{X_s}, \widetilde{\Pbb}_{Y_s}\right)^{2q} ds \right)^{\frac{1}{2q}} + \left( \int_0^t \Kcal\left(\Pbb_{X_s},\widetilde{\Pbb}_{Y_s} \right)^{2q} ds \right)^{\frac{1}{q}}.
	\end{align*}
	\end{small}
	
	Assume $\int_0^t \Kcal(\Pbb_{X_s}, \widetilde{\Pbb}_{Y_s})^{2q} ds \geq 1$. Then, 
	\begin{align*}
		\Kcal\left(\Pbb_{X_t}, \widetilde{\Pbb}_{Y_t}\right)^q \lesssim \int_0^t \Kcal\left(\Pbb_{X_s}, \widetilde{\Pbb}_{Y_s}\right)^{2q} ds.
	\end{align*}
	In the case $0\leq \int_0^t \Kcal(\Pbb_{X_s}, \widetilde{\Pbb}_{Y_s})^{2q} ds < 1$, we get
	\begin{align*}
		\Kcal\left(\Pbb_{X_t}, \widetilde{\Pbb}_{Y_t}\right)^{2q} \lesssim \int_0^t \Kcal\left(\Pbb_{X_s}, \widetilde{\Pbb}_{Y_s}\right)^{2q} ds.
	\end{align*}
	Next we show that $t \mapsto \Kcal\left(\Pbb_{X_t}, \widetilde{\Pbb}_{Y_t}\right)$ is continuous. Since $t \mapsto X_t$ and $t \mapsto Y_t$ are almost surely continuous, we immediately get that $t \mapsto \Pbb_{X_t}$ and $t \mapsto \widetilde{\Pbb}_{Y_t}$ are weakly continuous. Furthermore, it can be shown as in the proof of \Cref{thm:weakSolution} that $\left\lbrace \Pbb_{X_t}: t\in [0,T] \right\rbrace$ and $\left\lbrace \widetilde{\Pbb}_{Y_t}: t\in [0,T] \right\rbrace$ are relatively compact with respect to the Kantorovich-Rubinstein metric and consequently, that $t \mapsto \Kcal\left(\Pbb_{X_t}, \widetilde{\Pbb}_{Y_t}\right)$ is continuous. Hence, using Grönwall's inequality in the first case and a non-linear Grönwall type inequality by Stachurska \cite[Theorem 25]{dragomir2003some} in the second, yields $\Kcal\left(\Pbb_{X_u}, \widetilde{\Pbb}_{Y_t}\right) = 0$ for all $t\in[0,T]$ and thus the proof is complete.	
\end{proof}

\section{Strong Solutions and Pathwise Uniqueness}\label{sec:strongSolution}

	In this section we examine under which assumptions MKV equation \eqref{eq:mfsde} has a pathwisely unique strong solution. Therefore, we first recall the definitions of a strong solution and pathwise uniqueness.

	\begin{definition}\label{def:MFStrS}
	A \emph{strong solution} of MKV equation \eqref{eq:mfsde} is a weak solution $(\Omega, \Fcal, \Fbb^\Bb, \Pbb, \Bb, X)$ where $\Fbb^\Bb$ is the filtration generated by the weighted cylindrical fractional Brownian motion $\Bb$ and augmented with the $\Pbb$-null sets.
\end{definition}

\begin{definition}
	We say a weak solution $(\Omega, \Fcal, \Fbb, \Pbb, \Bb, X)$ of MKV equation \eqref{eq:mfsde} is \emph{pathwisely unique}, if for any other weak solution $(\Omega, \Fcal, \Fbb, \Pbb, \Bb, Y)$ on the same stochastic basis with the same initial condition $X_0 = Y_0$,
	\begin{align*}
		\Pbb\left( \forall t\geq 0: X_t = Y_t \right) = 1.
	\end{align*}
\end{definition}

\begin{remark}
	In the following we speak of a unique solution, if the solution is unique in law and pathwisely unique.
\end{remark}

	Provided that a weak solution of MKV equation \eqref{eq:mfsde} exists, the task of proving the existence of a strong solution becomes a problem in the field of SDEs. More precisely, the difference between a weak and a strong solution lies in the measurability with respect to the filtration of the driving noise. Since the dependence on the law is mere deterministic, it does not effect adaptedness of the solution. Therefore, the SDE
	\begin{align}\label{eq:ddSDE}
		Y_t = Y_0 + \int_0^t b^{\Pbb_X}(s,Y_s) dt + \Bb_t, \quad t\in [0,T],
	\end{align}
	can be considered, where $b^{\Pbb_X}(s,y) = b\left(s,y, \Pbb_{X_s}\right)$ and $(X_s)_{s\in[0,T]}$ is a weak solution of MKV equation \eqref{eq:mfsde}. For more details on this transition we refer the reader to \cite{Bauer_StrongSolutionsOfMFSDEs}. Subsequently we give a general result regarding strong solutions of MKV equation \eqref{eq:mfsde}.
	
	\begin{theorem}\label{thm:strong}
		Suppose the assumptions of \Cref{thm:weakSolution} are fulfilled and SDE \eqref{eq:ddSDE} has a unique strong solution $(Y_t)_{t\in [0,T]}$. Then, MKV equation \eqref{eq:mfsde} has a strong solution. More precisely, any weak solution $(X_t)_{t\in[0,T]}$ of MKV equation \eqref{eq:mfsde} is a strong solution. If in addition $\sup_{k\in I_+} H_k < 1$ and condition \eqref{eq:lipschitzUniformly} is fulfilled, the solution of MKV equation \eqref{eq:mfsde} is unique.
	\end{theorem}
	\begin{proof}
		Due to \Cref{thm:weakSolution} there exists a weak solution $X$ of MKV equation \eqref{eq:mfsde}. Moreover, $X$ can be seen as a weak solution of the associated SDE \eqref{eq:ddSDE}. Since SDE \eqref{eq:ddSDE} has a unique strong solution $Y$, i.e. in particular $Y$ is a weak solution which is unique in law, we have that $X$ and $Y$ have the same law. Thus, equations \eqref{eq:mfsde} and \eqref{eq:ddSDE} coincide and $Y$ is a strong solution of MKV equation \eqref{eq:mfsde}. \par
		Under the additional assumptions $\sup_{k\in I_+} H_k < 1$ and condition \eqref{eq:lipschitzUniformly}, we know by \Cref{thm:weakUniqueness} that the  weak solution $X$ of MKV equation \eqref{eq:mfsde} is unique in law. Consequently, there exists a unique associated SDE \eqref{eq:ddSDE}, which has by assumption the unique strong solution $Y$. In particular, $Y$ is also a strong solution of MKV equation \eqref{eq:mfsde} due to the first part. Since the associated SDE is uniquely determined, the pathwise uniqueness of a solution to SDE \eqref{eq:ddSDE} transfers to the solution of MKV equation \eqref{eq:mfsde}. Thus, $Y$ is the unique strong solution of MKV equation \eqref{eq:mfsde}.
	\end{proof}
	
	In the following we link \Cref{thm:strong} to results in the literature on the existence of strong solutions of SDEs. We start with a corollary in the infinite-dimensional case applying the result of \cite{banos2019infiniteSDE}. Subsequently, we consider the finite-dimensional case applying the result of \cite{nualart2002regularization}.
	
	\begin{corollary}
		Assume $I_0 \cup I_+ = \emptyset$, $\sum_{k\in I_-} H_k < \frac{1}{6}$, and $\sup_{k\in I_-} H_k < \frac{1}{12}$. Let $b:[0,T]\times \Hcal \times \Pcal_1(\Hcal) \to \Hcal$ be a measurable function fulfilling the Lipschitz condition \eqref{eq:lipschitzUniformly} and for which there exist sequences $C \in \ell^1$ and $D \in \ell^1$ such that for every $k\geq 1$ \vspace{0.2cm}
\begin{align*}
	\sup_{y\in \Hcal} \sup_{t\in [0,T]} \vert b_k( t, y, \mu ) \vert &\leq C_k \lambda_k, \text{ and} \\ 
	\sup_{d\geq 1} \int_{\Rbb^d} \sup_{t\in [0,T]} \vert b_k\left( t, \sqrt{Q} \sqrt{\Kcal} \tau^{-1} y, \mu \right) \vert dy &\leq D_k \lambda_k,
\end{align*}
	where $y=(y_1, \dots, y_d)$ and $\Kcal: \Hcal \to \Hcal$ is defined by
	\begin{align*}
		\Kcal x = \sum_{k\geq 1} \kfrak_{H_k} x^{(k)} e_k, ~x \in \Hcal,
	\end{align*}
	for $\lbrace \kfrak_{H_k} \rbrace_{k\geq 1}$ being the local non-determinism constant of $\lbrace B^{H_k} \rbrace_{k\geq 1}$, i.e. a constant merely dependent on $H$ such that for every $t\in [0,T]$ and $0 < r \leq t$
	\begin{align*}
		\Var \left( B_{t}^{H}\left\vert B_{s}^{H}: \left\vert t-s\right\vert \geq r \right. \right) \geq \kfrak_{H} r^{2H}.
	\end{align*}	
	Then, MKV equation \eqref{eq:mfsde} has a Malliavin differentiable unique strong solution.
	\end{corollary}
	\begin{proof}
		The result is an immediate consequence of \Cref{thm:strong} and \cite[Theorem 4.11]{banos2019infiniteSDE}.
	\end{proof}
	
	Consider now the one-dimensional real-valued MKV equation 
	\begin{align}\label{eq:McKddim}
		X_t = x + \int_0^t b \left( s, X_s, \Pbb_{X_s} \right) ds + B_t^H, \quad t \in [0,T],
	\end{align}
	where $b: [0,T] \times \Rbb \times \Pcal_1(\Rbb) \to \Rbb$ and $B_t^H$ one-dimensional fractional Brownian motion with Hurst parameter $H$.
	
	\begin{corollary}
		Let $b: [0,T] \times \Rbb \times \Pcal_1(\Rbb) \to \Rbb$ be a bounded measurable function. If $H> 1/2$ suppose that
		\begin{align*}
			\vert b (t,x,\mu) - b (s,y,\nu) \vert \leq C \left( \vert t-s \vert^{\gamma} + \vert x-y \vert^{\alpha} + \Kcal(\mu, \nu)^{\beta} \right),
		\end{align*}
		where $C>0$, $\gamma > H - \frac{1}{2}$, $2 \geq \alpha > 2H - 1$, and $\beta > H - \frac{1}{2}$, and if $H \leq 1/2$ suppose that for every $\mu \in \Ccal([0,T];\Pcal_1(\Rbb))$ and every $\varepsilon>0$ there exists $\delta>0$ such that for all $\nu \in \Ccal([0,T];\Pcal_1(\Rbb))$
		\begin{align*}
			\sup_{t\in [0,T]} \Kcal(\mu_t,\nu_t) < \delta ~\Rightarrow \sup_{t\in[0,T], ~ y\in \Hcal}\left\vert b (t,y,\mu_t) - b (t,y,\nu_t) \right\vert < \varepsilon.
		\end{align*}
		Then, MKV equation \eqref{eq:McKddim} has a strong solution. If in addition condition \eqref{eq:lipschitzUniformly} is fulfilled, the solution is unique.
	\end{corollary}
	\begin{proof}
		This result is a direct consequence of \cite{nualart2002regularization} together with \Cref{thm:weakSolution} and \Cref{thm:weakUniqueness}, respectively.
	\end{proof}

\appendix
\begin{footnotesize}
	\bibliography{literatureTH}
	\bibliographystyle{abbrv}
\end{footnotesize}
\bigskip
\rule{\textwidth}{1pt}

\end{document}